\documentclass[11pt,a4paper]{article}

\usepackage{a4}

\usepackage[utf8]{inputenc}
\usepackage[english]{babel}

\usepackage{amsmath, amsthm}
\usepackage{amsfonts,amssymb}
\usepackage{mathrsfs,upgreek}
\usepackage{dsfont,xfrac}
\usepackage{thmtools,yhmath}
\usepackage{authblk}
\usepackage{geometry}
\usepackage[dvipsnames]{xcolor}

\usepackage{hyperref}

\usepackage{chngcntr}

\newcommand{\N}{\mathbb{N}}

\newcommand{\Q}{\mathbb{Q}}
\newcommand{\R}{\mathbb{R}}

\newcommand{\rmC}{{\rm C}}

\newcommand{\1}{\mathds{1}}

\newcommand{\Fc}{\mathcal{F}}
\newcommand{\Ec}{\mathcal{E}}
\newcommand{\Gc}{\mathcal{G}}

\newcommand{\Bc}{\mathcal{B}}

\newcommand{\Uc}{\mathcal{U}}

\renewcommand{\mid}{~\middle | ~}
\newcommand{\eps}{\varepsilon}
\newcommand{\mycdot}{\cdot}
\newcommand{\cv}[2][]{\underset{#2}{\overset{#1}{\longrightarrow}}}
\newcommand\quotient[2]{
        \mathchoice
            {
                \text{\raise1ex\hbox{$#1$}\Big/\lower1ex\hbox{$#2$}}%
            }
            {
                #1\,/\,#2
            }
            {
                #1\,/\,#2
            }
            {
                #1\,/\,#2
            }
}
\newcommand{\cadlag}{cadlag }

\newcommand{\Pcal}{\mathcal{P}}
\newcommand{\D}{\mathbb{D}}

\newcommand{\loc}{\text{loc}}
\newcommand{\Dloc}{\D_\loc}
\newcommand{\Dexp}{\D_{\text{exp}}}
\renewcommand{\d}{{\rm d}}
\newcommand{\Pbf}{\mathbf{P}}
\newcommand{\Ebf}{\mathbf{E}}

\newcommand{\esssup}{{\rm \text{esssup}}}

\makeatletter
\DeclareRobustCommand\widecheck[1]{{\mathpalette\@widecheck{#1}}}
\def\@widecheck#1#2{%
    \setbox\z@\hbox{\m@th$#1#2$}%
    \setbox\tw@\hbox{\m@th$#1%
       \widehat{%
          \vrule\@width\z@\@height\ht\z@
          \vrule\@height\z@\@width\wd\z@}$}%
    \dp\tw@-\ht\z@
    \@tempdima\ht\z@ \advance\@tempdima2\ht\tw@ \divide\@tempdima\thr@@
    \setbox\tw@\hbox{%
       \raise\@tempdima\hbox{\scalebox{1}[-1]{\lower\@tempdima\box
\tw@}}}%
    {\ooalign{\box\tw@ \cr \box\z@}}}
\makeatother

\usepackage[many]{tcolorbox}
\newcounter{myhypo}

\newtcolorbox{hypo}[1]{
  breakable,
  enhanced,
  top=0pt,
  bottom=0pt,
  nobeforeafter,
  colback=white,
  boxrule=0pt,
  arc=0pt,
  right=40pt,
  left=20pt,
  outer arc=0pt,
  overlay={
    \node[inner sep=0pt,anchor=east] 
    at (frame.east) 
    {(#1)};
  },
}

\makeatletter
\newcommand{\pushright}[1]{\ifmeasuring@#1\else\omit\hfill$\displaystyle#1$\fi\ignorespaces}
\newcommand{\pushleft}[1]{\ifmeasuring@#1\else\omit$\displaystyle#1$\hfill\fi\ignorespaces}
\makeatother

\theoremstyle{plain}
\newtheorem{theorem}{Theorem}[section]
\newtheorem{lemma}[theorem]{Lemma}
\newtheorem{proposition}[theorem]{Proposition}
\newtheorem{corollary}[theorem]{Corollary}

\theoremstyle{definition}
\newtheorem{definition}[theorem]{Definition}

\newtheorem{remark-base}[theorem]{Remark}

\theoremstyle{remark}

\counterwithin*{equation}{section}

\newenvironment{remark}{\pushQED{\qed}\begin{remark-base}}{\popQED\end{remark-base}}

\begin{document}
\title{Local Skorokhod topology on the space of \cadlag processes}
\author{Mihai Gradinaru}
\author{Tristan Haugomat}
\affil{\small Institut de Recherche Math{\'e}matique de Rennes, Universit{\'e} de Rennes 1,\\Campus de Beaulieu, 35042 Rennes Cedex, France\\
\texttt{\small \{Mihai.Gradinaru,Tristan.Haugomat\}@univ-rennes1.fr}}
\date{}
\maketitle

{\small\noindent {\bf Abstract:}~We modify the global Skorokhod topology, on the 
space of \cadlag paths, by localising with respect to space variable, in order to include the eventual explosions. The tightness of families of probability measures on the paths space endowed with this local Skorokhod topology  is studied and a  characterization of Aldous type is obtained. The local and  global Skorokhod topologies are compared by using a time change transformation. A number of results in the paper should play an important role when studying Lévy-type processes with unbounded coefficients by martingale problem approach. }
{\small\noindent{\bf MSC2010 Subject Classification:}~Primary~ 60B10; Secondary~60J75, 54E70, 54A10}

 
\section{Introduction}

The study of \cadlag Lévy-type processes has been an important challenge during 
the last twenty years. This was due to the fact that phenomena like jumps and
unbounded coefficients of characteristic exponent (or symbol) should be taken in consideration in order to get more  realistic models. 

To perform a systematic study of this kind of trajectories one needs, on one hand, 
to consider the space of \cadlag paths with some  appropriate topologies, e.g. 
Skorokhod's topologies. 
On the other hand it was a very useful observation that a unified manner to tackle a 
lot of questions about large classes of processes is the martingale problem approach. 
Identifying tightness is an important step when studying sequences of distributions 
of processes solving associated martingale problems and the Aldous criterion is one 
of the most employed. 

The martingale problem approach was used for several situations: diffusion processes, 
stochastic differential equations driven by Brownian motion, Lévy processes, Poisson random measures (see, for instance, Stroock \cite{St75}, Stroock and Varadhan \cite{SV06}, Kurtz \cite{Ku11}...). Several  technical  hypotheses (for instance, 
entire knowledge of the generator, bounded coefficients hypothesis, assumptions 
concerning explosions ...) provide some 
limitation on the conclusions of certain results, in particular, on convergence results.

The present paper constitutes our first step in studying Markov processes with explosion and, in particular in the martingale problem setting. It contains the study 
of  the so-called local Skorokhod topology and of a time change transformation of 
\cadlag paths. The detailed study of the martingale problem, of Lévy-type processes
and of some applications will be presented elsewhere (see \cite{GH117}).

One of our motivations is that we wonder whether the solution of a well-posed 
martingale problem is continuous with respect to the initial distribution? 
The classical approach when one needs to take in consideration the explosion of 
the solution is to compactify the state space by one point, say $\Delta$, and to endow the \cadlag paths space 
by the Skorokhod topology (see for instance Ethier and Kurtz \cite{EK86}, Kallenberg 
\cite{Ka02}).  Unfortunately, this usual topology is not appropriate when we relax hypotheses on the martingale problem setting. 

The most simple example is provided by the differential equation 
\[\dot{x}_t=b(t,x_t),\quad t>0,\quad\mbox{starting from }\,x_0\in\R^d,\]
where $b:\R_+\times \R^d\to\R^d$ is a locally Lipschitz function. The unique maximal solution 
exists by setting $x_t=\Delta$,  after the explosion time. In general, for some $t>0$, the 
mapping $x_0\mapsto x_t$ is not continuous, and in particular $x_0\mapsto x_{\bullet}$ 
is not continuous for the usual (global) Skorokhod topology. As an illustration, let us 
consider
\[\dot{x}_t=(1-t)x_t^2,\quad t>0,\quad x_0\in\R.\]
To achieve the continuity of the mapping $x_0\mapsto x_{\bullet}$  our idea will be to localise 
the topology on the paths space, with respect to the time variable but also 
with respect to the space variable. More precisely, we need to consider uniform convergence 
until the exit time from some compact subset of $\R_+\times\R^d$. 

We adapt this idea to \cadlag paths 
by following a similar  approach as in Billingsley \cite{Bi99} and we get the local Skorokhod topology
which is weaker than the usual (global) Skorokhod topology. Then we describe the 
compactness and the tightness in connection with this topology. Furthermore, we state 
and prove a slight and new, at our knowledge, improvement of the Aldous criterion, which becomes an equivalence
in our setting.

Another novelty of our paper is the employ of a time change transformation 
(see for instance Ethier and Kurtz \cite{EK86}, pp. 306-311) to compare the local Skorokhod topology 
with the usual  (global) Skorokhod topology. Roughly speaking, the time change 
of the \cadlag path $x$ by the positive continuous function $g$ is 
$(g\mycdot x)_t:=x_{\tau_t}$ with  $\tau_t$ the unique  solution starting from $0$ 
of $\dot{\tau}_t:=g(x_{\tau_t})$.

Our paper is organised as follows: the following section is mainly devoted to the study 
of the local Skorokhod topology on spaces of \cadlag paths: the main result is a 
tightness criterion.  Properties of the time change mapping, in particular the 
continuity, and the connection between the local and global Skorokhod  topologies 
are described in Section 3. The last section contains technical proofs, based on local 
Skorokhod metrics, of results stated in \S2.

\section{Paths spaces}
\subsection{Local spaces of \cadlag  paths}
Let $S$ be a locally compact Hausdorff space with countable base. This topological feature is equivalent with the fact that $S$ could be endowed with a metric which is separable and have compact balls,
so $S$ is a Polish space. Take $\Delta\not\in S$, and we will denote by 
$S^\Delta\supset S$ 
the one-point compactification of $S$, if $S$ is not compact, 
or the topological sum $S\sqcup\{\Delta\}$, if $S$ is compact  (so $\Delta$ is an isolated point).  Clearly, $S^\Delta$ is a compact Hausdorff space with countable base which could 
be also endowed with a metric. This latter metric will be used to construct various useful functions, compact and open subsets. 

For any topological space $A$ and any subset $B\subset\R$, we will denote by $\rmC(A,B)$ the set of continuous functions from $A$ to $B$, and by $\rmC_b(A,B)$ its subset of bounded continuous functions. We will abbreviate $\rmC(A):=\rmC(A,\R)$ and $\rmC_b(A):=\rmC_b(A,\R)$. All along the paper we will denote $C\Subset A$ for a subset $C$ which is compactly embedded in $A$.

We start with the definition our  spaces of trajectories:
\begin{definition}[Spaces of \cadlag paths]\label{defCADLAG}
Define 
the space of exploding \cadlag paths
\begin{equation*}
\Dexp(S):=\left\{ x:[0,T_{\max})\rightarrow S\mid \begin{array}{l}
0\leq T_{\max}\leq\infty,\\
\forall t_0\in[0,T_{\max})\quad x_{t_0}=\displaystyle\lim_{t\downarrow t_0}x_t,\\
\forall t_0\in(0,T_{\max})\quad x_{t_0-}:=\displaystyle\lim_{t\uparrow t_0}x_t~\text{exists in }S\\
\end{array}\right\}\,.
\end{equation*}
For a path from  $\Dexp(S)$, $x:[0,T_{\max})\rightarrow S$, we will denote  $\xi(x):=T_{\max}$. We identify $\Dexp(S)$ with a subset of $(S^\Delta)^{\R_+}$ by using the mapping
\[\begin{array}{ccc}
\Dexp(S) & \hookrightarrow & (S^\Delta)^{\R_+}\\
x & \mapsto & (x_t)_{t\geq 0}
\end{array}\hspace{1.5cm}\text{with }\quad x_t:=\Delta\quad \text{ if }\quad t\geq\xi(x).
\]
We define the local \cadlag space as the subspace
\begin{equation}\label{Dloc}
\Dloc(S):=\left\{x\in\Dexp(S)\mid \xi(x)\in(0,\infty)\text{ and }\{x_s\}_{s<\xi(x)}\Subset S\text{ imply } x_{\xi(x)-}\text{ exists}\right\}.
\end{equation}
We also introduce the global \cadlag space as the  subspace of $\Dloc(S)$
\[
\D(S):=\left\{x\in\Dloc(S)\mid \xi(x)=\infty\right\}\subset S^{\R_+}.
\]
We will always denote by $X$ the canonical process on $\Dexp(S)$, $\Dloc(S)$ 
and $\D(S)$ without danger of confusion.  We endow each of $\Dexp(S)$, 
$\Dloc(S)$ and $\D(S)$ with a $\sigma$-algebra $\Fc:=\sigma(X_s,~0\leq s <\infty)$ 
and a filtration $\Fc_t:=\sigma(X_s,~0\leq s \leq t)$. We will always skip the argument 
$X$ for the explosion time $\xi(X)$ of the  canonical process.
\end{definition}
The following result provides an useful class of measurable mappings:
\begin{proposition}\label{propMes}
For $t_0\in\R_+$, the mapping
\[\begin{array}{ccc}
\Dexp(S)\times [0,t_0]&\to&S^\Delta\\
(x,t)&\mapsto&x_t
\end{array}\]
is $\Fc_{t_0}\otimes \Bc([0,t_0])$-measurable. For $t_0\in\R_+^*$, the set
\[
A:=\left\{(x,t)\in\Dexp(S)\times (0,t_0]\mid x_{t-}\text{ exists in }S^\Delta\right\}
\]
belongs to $\Fc_{t_0-}\otimes \Bc((0,t_0])$ and the mapping
\[\begin{array}{ccc}
A&\to&S^\Delta\\
(x,t)&\mapsto&x_{t-}
\end{array}\]
is $\Fc_{t_0-}\otimes \Bc((0,t_0])$-measurable.
For $U$  an open subset of $S$ and for $t_0\in\R_+$, the set
\[
B:=\left\{(x,s,t)\in\Dexp(S)\times[0,t_0]^2\mid \{x_u\}_{s\wedge t\leq u< s\vee t}\Subset U\right\}
\]
belongs to $\Fc_{t_0-}\otimes \Bc([0,t_0])^{\otimes 2}$ and the mapping
\[\begin{array}{ccc}
B\times\rmC(U)&\to&\R\\
(x,s,t,h)&\mapsto&\displaystyle\int_s^th(x_u)\d u
\end{array}\]
is $\Fc_{t_0-}\otimes \Bc([0,t_0])^{\otimes 2}\otimes\Bc(\rmC(U))$-measurable.
\end{proposition}
Before proving this proposition we state a corollary which give an useful class of stopping times:
\begin{corollary}\label{corST}
For any $(\Fc_t)$-stopping time $\tau_0$, $\Uc$ an open subset of $S^2$, $h\in\rmC(\Uc,\R_+)$ a continuous function and $M:\Dexp(S)\to[0,\infty]$ a $\Fc_{\tau_0}$-measurable map, the mapping
\[
\tau:=\inf\Big\{t\geq \tau_0\,|\,\{(X_{\tau_0},X_s)\}_{\tau_0\leq s\leq t}\not\Subset \Uc\text{ or }\int_{\tau_0}^th(X_{\tau_0},X_s)\d s\geq M\Big\}
\]
is a $(\Fc_t)$-stopping time. 
In particular, $\xi$ is a stopping time. Furthermore, if $U\subset S$ is an open subset, 
\begin{equation}\label{eqtauU}
\tau^U:=\inf\left\{t\geq 0\mid X_{t-}\not\in U\text{ or }X_t\not\in U\right\}\leq\xi
\end{equation}
is a stopping time.
\end{corollary}
\begin{proof}[Proof of Corollary \ref{corST}]
For each $t\geq 0$, using Proposition \ref{propMes} it is straightforward to obtain that
\[
Y:=\left\{\begin{array}{ll}
\displaystyle -1&\text{if }\tau_0>t,\\
\int_{\tau_0}^th(X_{\tau_0},X_s)\d s&\text{if }{\tau_0}\leq t\text{ and }\{(X_{\tau_0},X_s)\}_{{\tau_0}\leq s\leq t}\Subset \Uc,\\
\displaystyle\infty &\text{otherwise}.
\end{array}\right.
\]
is $\Fc_t$-measurable. Hence
\[
\{\tau\leq t\}=\{Y\geq M\}=\{Y\geq M\}\cap\{\tau_0\leq t\}\in\Fc_t,
\]
so $\tau$ is a $(\Fc_t)$-stopping time.
\end{proof}

\begin{proof}[Proof of Proposition \ref{propMes}]
Let $d$ be a complete metric for the topology of $S$, note that
\[
A = \bigcap_{\eps\in\Q_+^*}\bigcup_{\delta\in\Q_+^*}\bigcap_{q_1,q_2\in\Q_+\cap [0,t_0)}\left\{q_1,q_2\in[t-\delta,t)\Rightarrow d(x_{q_1},x_{q_2})\leq\eps\right\}
\]
so $A$ belongs to $\Fc_{t_0-}\otimes \Bc((0,t_0])$.
It is clear that
for each $n\in\N$
\[\begin{array}{ccc}
A&\to&S^\Delta\\
(x,t)&\mapsto&x_{\frac{t_0}{n+1}\left\lfloor\frac{nt}{t_0}\right\rfloor}
\end{array}\]
is $\Fc_{t_0-}\otimes \Bc((0,t_0])$-measurable, where $\lfloor r\rfloor$ denotes the integer part of the real number $r$. Letting $n\to\infty$ we obtain that $(x,t)\mapsto x_{t-}$ is $\Fc_{t_0-}\otimes \Bc((0,t_0])$-measurable. The proof is similar for $(x,t)\mapsto x_t$.
To prove that $B$ is measurable, let $(K_n)_{n\in\N}$  be an increasing sequence of compact subsets of $U$ such that $U=\bigcup_nK_n$. Then
\begin{align*}
B&=\bigcup_{n\in\N}\left\{(x,s,t)\in\Dexp(S)\times[0,t_0]^2\mid \{x_u\}_{s\wedge t\leq u< s\vee t}\subset K_n\right\}\\
&=\bigcup_{n\in\N}\bigcap_{\substack{q\in\Q_+\\q<t_0}}\left\{(x,s,t)\in\Dexp(S)\times[0,t_0]^2\mid s\wedge t\leq q< s\vee t\Rightarrow x_q\in K_n\right\},
\end{align*}
so $B\in\Fc_{t_0-}\otimes \Bc([0,t_0])^{\otimes 2}$. To verify the last part, let us note that for $n\in\N^*$ the mapping 
from $B\times\rmC(U)$
\[
(x,s,t,h)\mapsto\frac{{\rm sign}(t-s)}{n}\sum_{i=0}^{n-1}h(x_{\frac{it_0}{n}})\1_{s\wedge t\leq \frac{it_0}{n}<s\vee t}
\]
is $\Fc_{t_0-}\otimes \Bc([0,t_0])^{\otimes 2}\otimes\Bc(\rmC(U))$-measurable so, letting $n\to\infty$, 
the same thing is true for the mapping
\[\begin{array}{ccc}
B\times\rmC(U)&\to&\R\\
(x,s,t,h)&\mapsto&\displaystyle\int_s^th(x_u)\d u.
\end{array}\]
\end{proof}

We end this section by recalling the definition of a Markov family:
\begin{definition}[Markov family]
Let $(\Gc_t)_{t\geq 0}$ be a filtration containing $(\Fc_t)_{t\geq 0}$. A family of probability measures $(\Pbf_a)_{a\in S}\in\Pcal(\Dexp(S))^S$ is called $(\Gc_t)_t$-Markov if
\begin{enumerate}
\item[a)] for any $B\in\Fc$: $a\mapsto\Pbf_a(B)$ is measurable,
\item[b)] for any $a\in S$: $\Pbf_a(X_0=a)=1$,
\item[c)] for any $a\in S$, $B\in\Fc$ and $t_0\in\R_+$: $\Pbf_a\left((X_{t_0+t})_t\in B\mid\Gc_{t_0}\right)=\Pbf_{X_{t_0}}(B)$, $\Pbf_a$-almost surely, where $\Pbf_\Delta$ is the unique element of $\Pcal(\Dexp(S))$ such that $\Pbf_\Delta(\xi=0)=1$.
\end{enumerate}
If the last property is also satisfied by replacing $t_0$ with any $(\Gc_t)_t$-stopping time,
the family of probability measures  is called $(\Gc_t)_t$-strong Markov.
\end{definition}

\begin{remark}
1) If $\Gc_t=\Fc_t$ we just say that the family is (strong) Markov.\\
2) If $\nu$ is a measure on $S^\Delta$ we set $\Pbf_\nu:=\int\Pbf_a\nu(\d a)$. Then the distribution of $X_0$ under $\Pbf_\nu$ is $\nu$, 
and $\Pbf_\nu$ satisfies the (strong) Markov property.
\end{remark}

\subsection{Local Skorokhod topology}

To simplify some statements, in this section we will consider a metric $d$ on $S$. However, we will prove that the construction does not depend on a particular choice of $d$.

To describe the convergence of a sequence $(x^k)_{k\in\N}
\subset\Dloc(S)$ for our topology on $\Dloc(S)$,
we need to introduce the following two spaces: we denote by $\widetilde{\Lambda}$ the space of increasing bijections from $\R_+$ to $\R_+$, and by $\Lambda\subset\widetilde{\Lambda}$ the space of increasing bijections $\lambda$ with $\lambda$ and $\lambda^{-1}$ locally Lipschitz. For $\lambda\in\widetilde{\Lambda}$ and $t\in\R_+$ we denote
\begin{equation}\label{tnorme}
\|\lambda-{\rm id}\|_t:=\sup_{0\leq s\leq t}|\lambda_s -s|=\|\lambda^{-1}-{\rm id}\|_{\lambda_t}.
\end{equation}
For $\lambda\in\Lambda$, let $\dot{\lambda}\in{\rm L}_{\rm loc}^\infty(\d s)$ be the density of $\d \lambda$ with respect to the Lebesgue measure. This density is non-negative and  locally bounded below, and for $t\in\R_+$ denote
\[
\|\log\dot{\lambda}\|_{t}:=\esssup_{0\leq s\leq t}\|\log\dot{\lambda}_s\|=\sup_{0\leq s_1<s_2\leq t}\left|\log\frac{\lambda_{s_2}-\lambda_{s_1}}{s_2-s_1}\right|=\left\|\log\left(\frac{\d\lambda^{-1}_s}{\d s}\right)\right\|_{\lambda_t}.
\]

The proofs of the following theorems 
use the strategy developed in \S 12, pp. 121-137 from \cite{Bi99}, and are postponed to Section \ref{secSkoMet}.

\begin{theorem}[Local Skorokhod topology]\label{thmLocSkoTop}
There exists a unique Polish topology on $\Dloc(S)$, called the local Skorokhod topology, such that a sequence $(x^k)_{k\in\N}$ converges 
to $x$ for this topology if and only if there exists a sequence $(\lambda^k)_{k\in\N}$ in $\Lambda$ such that
\begin{itemize}
\item either $\xi(x)<\infty$ and $\{x_s\}_{s<\xi(x)}\Subset S$: $\lambda^k_{\xi(x)}\leq\xi(x^k)$ for $k$ large enough and
\[
\sup_{s<\xi(x)}d(x_s,x^k_{\lambda^k_s})\cv{}0,\hspace{1cm}
x^k_{\lambda^k_{\xi(x)}}\cv{}\Delta,\hspace{1cm}\|\log\dot{\lambda}^k\|_{\xi(x)}\cv{}0,\quad\mbox{ as }k\to\infty,
\]
\item or $\xi(x)=\infty$ or $\{x_s\}_{s<\xi(x)}\not\Subset S$: for all $t<\xi(x)$, for $k$ large enough $\lambda^k_t<\xi(x^k)$ and
\[
\sup_{s\leq t}d(x_s,x^k_{\lambda^k_s})\cv{}0,\hspace{2cm}\|\log\dot{\lambda}^k\|_t\cv{}0,\quad\mbox{ as }k\to\infty.
\]
\end{itemize}
The local Skorokhod topology does not depend on the metric $d$, but only on the topology of $S$. Moreover the Borel $\sigma$-algebra $\Bc(\Dloc(S))$ coincides with the $\sigma$-algebra $\Fc$.
\end{theorem}
\begin{theorem}\label{th2ndCarCvg}
The local Skorokhod topology is also described by a similar characterisation with $\lambda^k\in\Lambda$ and $\|\log\dot{\lambda}^k\|$ replaced, respectively, by $\lambda^k\in\widetilde{\Lambda}$ and  $\|\lambda^k -{\rm id}\|$.
\end{theorem}
\begin{remark}
These convergence conditions may be summarised as: a sequence $(x_k)_k$ converges to $x$ for the local Skorokhod topology if and only if there exists a sequence $(\lambda^k)_k$ in $\Lambda$ satisfying that for any $t\in\R_+$ such that $\{x_s\}_{s<t}\Subset S$, for $k$ large enough $\lambda^k_t\leq\xi(x^k)$ and
\[
\sup_{s<t}d(x_s,x^k_{\lambda^k_s})\cv{}0,\quad x^k_{\lambda^k_t}\cv{}x_t,\quad\|\log\dot{\lambda}^k\|_t\cv{}0,\quad\mbox{ as }k\to\infty.
\qedhere\]
\end{remark}
 We point out that the topology on $\Dloc(S)$ does not depend on the metric $d$ of $S$ and this is a consequence of the fact that two metrics on a compact set are uniformly equivalent (cf. Lemma \ref{lmUU} below).
We recover the classical Skorokhod topology on $\D(S)$, which is described, for instance, in \S16 pp. 166-179 from \cite{Bi99}.
\begin{corollary}[Global Skorokhod topology]\label{corGlobSkoTop}
The trace topology from $\Dloc(S)$ to $\D(S)$ will be called the global Skorokhod topology. It is a Polish
topology and a sequence $(x^k)_k$ converges to $x$ for this topology if and only if there exists a sequence $(\lambda^k)_k$ in $\Lambda$ such that for all $t\geq 0$
\[
\sup_{s\leq t}d(x_s,x^k_{\lambda^k_s})\cv{}0,\hspace{2cm}\|\log\dot{\lambda}^k\|_t\cv{}0, \quad\mbox{ as }k\to\infty.
\]
Once again the global Skorokhod topology does not depend on the metric $d$, but only on the topology of $S$. Moreover the Borel $\sigma$-algebra $\Bc(\D(S))$ coincides with $\Fc$.
\end{corollary}
\begin{proof}
The only thing to prove is the fact that this topology is Polish. This is true because
\[
\D(S)=\bigcap_{n\in\N}\left\{x\in\Dloc(S)\mid\xi(x)>n\right\},
\]
so $\D(S)$ is a countable intersection of open subsets of $\Dloc(S)$.
\end{proof}
\begin{remark}
Again, as in Theorem \ref{th2ndCarCvg}, the global Skorokhod topology can be described by the similar characterisation with $\lambda^k\in\Lambda$ replaced by $\lambda^k\in\widetilde{\Lambda}$ and $\|\log\dot{\lambda}^k\|$ by $\|\lambda^k -{\rm id}\|$.
\end{remark}
\begin{remark}\label{rkOthLocSko}
1) We may also recover the Skorokhod topology on the set of \cadlag paths from $[0,t]$ to $S$, $\D([0,t],S)$. Let $S\sqcup\{\Delta\}$ be 
the topological space such that $\Delta$ is an isolated point. The Skorokhod topology on $\D([0,t],S)$ is the pullback topology by the 
injective mapping with closed range
\[\begin{array}{ccc}
\D([0,t],S)&\hookrightarrow&\D(S\sqcup\{\Delta\})\times S\\
x&\mapsto&\left(\left(\begin{array}{ll}
x_s&\text{if }s<t\\
\Delta&\text{if }s\geq t
\end{array}\right)_{s\geq 0},x_t\right).
\end{array}\]
2) To get a topology on $\Dexp(S)$, we may proceed as follows: if $\widetilde{d}$ is a metric on $S^\Delta$, then the topology on $\Dexp(S)$ is the pullback topology by the injective mapping with closed range
\[\begin{array}{ccc}
\Dexp(S)&\hookrightarrow&\Dloc(S\times\R_+)\\
x&\mapsto&\left(x_s,\sup_{u\geq s}\big(\widetilde{d}(x_u,\Delta)+u-s\big)^{-1}\right)_{0\leq s <\xi(x)}.
\end{array}\]
It can be proved that it is a Polish topology and  a sequence $(x^k)_k$ converges to $x$ for this topology if and only if there exists a sequence $(\lambda^k)_k$ in $\Lambda$ such that for all $t<\xi(x)$, for $k$ large enough $\lambda^k_t<\xi(x^k)$ and
\[
\sup_{s\leq t}d(x_s,x^k_{\lambda^k_s})\cv{}0,\quad\|\log\dot{\lambda}^k\|_t\cv{}0,\quad\mbox{ as }k\to\infty,
\]
and if $\xi(x)<\infty$ and $\{x_s\}_{s<\xi(x)}\Subset S$
\[
x^k_{\lambda^k_{\xi(x)}}\cv{}\Delta,\quad\mbox{ as }k\to\infty.
\]
\end{remark}
We are now interested to characterise the sets of $\Dloc(S)$ which are compact and also to obtain a criterion for the tightness of a subset of probability measures in $\Pcal\left(\Dloc(S)\right)$. 
For $x\in \Dexp(S)$, $t\geq 0$, $K\subset S$ compact and $\delta>0$, define
\begin{align}\label{eqDefOme}
\omega^\prime_{t,K,x}(\delta):=\inf\left\{\sup_{\substack{0\leq i<N\\t_i\leq s_1,s_2<t_{i+1}}}d(x_{s_1},x_{s_2})\mid\begin{array}{l}
N\in\N,~0=t_0<\cdots <t_N\leq\xi(x)\\
(t_N,x_{t_N})\not\in [0,t]\times K\\
\forall 0\leq i<N:~t_{i+1}-t_i>\delta
\end{array}\right\}.
\end{align}
We give some properties of $\omega^\prime$.
\begin{proposition}\label{propMod}~
\begin{enumerate}
\item[i)]Consider $x\in\Dexp(S)$. Then $x$ belongs to $\Dloc(S)$ if and only if
\[
\forall t\geq 0,~\forall K\subset S\text{ compact,}\quad \omega^\prime_{t,K,x}(\delta)\cv{\delta\to 0}0.
\]
\item[ii)]For all $t\geq 0$, $K\subset S$ compact and $\delta>0$, the mapping
\[\begin{array}{ccc}
\Dloc(S)&\to&[0,+\infty]\\
x&\mapsto&\omega^\prime_{t,K,x}(\delta)
\end{array}\]
is upper semi-continuous
\end{enumerate}
\end{proposition}
\begin{proof}
Suppose that $x\in\Dloc(S)$ and let $t\geq 0$ and a compact $K\subset S$ be. There exists $T\leq \xi(x)$ such that $(T,x_T)\not\in [0,t]\times K$ and the limit $x_{T-}$ exists in $S$. Let $\eps>0$ be arbitrary and consider $I$ the set of times $s\leq T$ for which there exists a subdivision 
\[
0=t_0<\cdots<t_N = s
\]
such that
\[
\sup_{\substack{0\leq i<N\\t_i\leq s_1,s_2<t_{i+1}}}d(x_{s_1},x_{s_2})\leq \eps.
\]
It's clear that $I$ is an interval of $[0,T]$ containing $0$: set $t^*:=\sup I$. Since 
there is existence of the limit  $x_{t^*-}$, then $t^*\in I$, and, since $x$ is right-continuous, $t^*=T$. Hence $T\in I$ and there exists $\delta>0$ such that $\omega^\prime_{t,K,x}(\delta)\leq \eps$.

Conversely,  let's take $x\in\Dexp(S)$ such that $\xi(x)<\infty$, $\{x_s\}_{s<\xi(x)}\Subset S$ and
\[
\forall t\geq 0,~\forall K\subset S\text{ compact,}\quad \omega^\prime_{t,K,x}(\delta)\cv{\delta\to 0}0.
\]
We need to prove that the limit $x_{\xi(x)-}$ exists in $S$. Let $y_1,y_2$ 
be any two limits points of $x_s$, as $s\to\xi(x)$. We will prove that $y_1=y_2$.
Let $\eps>0$ be arbitrary . By taking $t=\xi(x)$ and $K=\overline{\{x_s\}}_{s<\xi(x)}$ in \eqref{eqDefOme} there exists a subdivision 
\[
0=t_0<\cdots<t_N=\xi(x)
\]
such that
\[
\sup_{\substack{0\leq i<N\\t_i\leq s_1,s_2<t_{i+1}}}d(x_{s_1},x_{s_2})\leq \eps.
\]
Replacing in the latter inequality the two sub-sequences tending toward $y_1$, $y_2$, we can deduce that $d(y_1,y_2)\leq\eps$, and letting $\eps\to 0$ we get $y_1=y_2$.

Let $(x^k)_k\subset\Dloc(S)$ be such that $x^k$ converges to $x\in\Dloc(S)$ and let $(\lambda^k)_k\subset\Lambda$ be such in Theorem~\ref{thmLocSkoTop}. We need to prove that
\[
\limsup_{k\to\infty}\omega^\prime_{t,K,x^k}(\delta)\leq \omega^\prime_{t,K,x}(\delta).
\]
We can suppose that $\omega^\prime_{t,K,x}(\delta)<\infty$. Let $\eps>0$ be arbitrary 
and consider a subdivision $0=t_0<\cdots<t_N\leq \xi(x)$  such that
\[
\sup_{\substack{0\leq i<N\\t_i\leq s_1,s_2<t_{i+1}}}d(x_{s_1},x_{s_2})\leq \omega^\prime_{t,K,x}(\delta)+\eps,
\]
$t_{i+1}>t_i+\delta$ and $(t_N,x_{t_N})\not\in [0,t]\times K$. If $t_N=\xi(x)$ and $\{x_s\}_{s<\xi(x)}\not\Subset S$, then we can find $\widetilde{t}_N$ such that  $t_{N-1}+\delta<\widetilde{t}_N<\xi(x)$ and $x_{\widetilde{t}_N}\not \in K$. 
We can suppose, possibly by replacing $t_N$ by $\widetilde{t}_N$, that 
\[
t_N=\xi(x)\mbox{ implies }\{x_s\}_{s<\xi(x)}\Subset S.
\]
Hence, for $k$ large enough, $\lambda^k_{t_N}\leq\xi(x^k)$ and
\[
\sup_{s<t_N}d(x_s,x^k_{\lambda^k_s})\cv{}0,\hspace{1cm}x^k_{\lambda^k_{t_N}}\cv{}x_{t_N},\hspace{1cm}\|\lambda^k-{\rm id}\|_{t_N}\cv{}0,\quad\mbox{ as }\,k\to\infty.
\]
We deduce that, for $k$ large enough, we have $0=\lambda^k_{t_0}<\cdots<\lambda^k_{t_N}\leq \xi(x^k)$, $\lambda^k_{t_{i+1}}>\lambda^k_{t_i}+\delta$, $(\lambda^k_{t_N},x^k_{\lambda^k_{t_N}})\not\in [0,t]\times K$, and moreover,
\begin{align*}
\sup_{\substack{0\leq i<N\\\lambda^k_{t_i}\leq s_1,s_2<\lambda^k_{t_{i+1}}}}d(x^k_{s_1},x^k_{s_2})
&\leq \sup_{\substack{0\leq i<N\\t_i\leq s_1,s_2<t_{i+1}}}d(x_{s_1},x_{s_2}) + 2\sup_{s<t_N}d(x_s,x^k_{\lambda^k_s})\\
&\leq\omega^\prime_{t,K,x}(\delta)+\eps+ 2\sup_{s<t_N}d(x_s,x^k_{\lambda^k_s})\cv{k\to\infty}\omega^\prime_{t,K,x}(\delta)+\eps.
\end{align*}
Therefore 
\[
\limsup_{k\to\infty}\omega^\prime_{t,K,x^k}(\delta)\leq \omega^\prime_{t,K,x}(\delta)+\eps,
\]
and we conclude by letting $\eps\to 0$.
\end{proof}
We can give now a characterisation of the relative compactness for the local Skorokhod topology:
\begin{theorem}[Compact sets of $\Dloc(S)$]\label{thmComp}
For any subset $D\subset\Dloc(S)$, $D$ is relatively compact if and only if
\begin{align}\label{eqThmComp}
\forall t\geq 0,~K\subset S\text{ compact,}\quad \sup_{x\in D}\omega^\prime_{t,K,x}(\delta)\cv{\delta\to 0}0.
\end{align}
\end{theorem}
The proof follows the strategy developed in \S 12 pp. 121-137 from \cite{Bi99} and it is postponed to Section  \ref{secSkoMet}.

We conclude this section with a version of the Aldous criterion of tightness:
\begin{proposition}[Aldous criterion]\label{propAldous}
Let $\Pcal$ be a subset of $\Pcal\left(\Dloc(S)\right)$. If for all $t\geq 0$, $\eps>0$, and an open subset $U\Subset S$, 
we have:
\[
\inf_{F\subset\Pcal}\sup_{\Pbf\in\Pcal\backslash F}\sup_{\substack{\tau_1\leq\tau_2\\ \tau_2\leq (\tau_1+\delta)\wedge t\wedge\tau^U}}\Pbf\big(\tau_1<\tau_2=\xi\text{ or } d(X_{\tau_1},X_{\tau_2})\1_{\{\tau_1<\xi\}}\geq\eps\big)\cv{\delta\to 0}0,
\]
 then $\Pcal$ is tight. Here the infimum is taken on all finite subsets $F\subset \Pcal$ and the supremum is taken on all stopping times $\tau_i$.
\end{proposition}
In fact we will state and prove an improvement of the Aldous criterion which becomes an equivalence:
\begin{theorem}[Tightness for $\Dloc(S)$]\label{thmTension}
For any subset $\Pcal\subset\Pcal\left(\Dloc(S)\right)$, the following assertions are equivalent:
\begin{enumerate}
\item\label{itemPropTen1} $\Pcal$ is tight,
\item\label{itemPropTen2} for all $t\geq 0$, $\eps >0$ and $K$ a compact set we have
\[
\sup_{\Pbf\in\Pcal}\Pbf\left(\omega^\prime_{t,K,X}(\delta)\geq\eps\right)\cv{\delta\to 0}0,
\]
\item\label{itemPropTen3} for all $t\geq 0$, $\eps>0$, and open subset $U\Subset S$ we have:
\[
\alpha(\eps,t,U,\delta):=\sup_{\Pbf\in\Pcal}\sup_{\substack{\tau_1\leq\tau_2\leq\tau_3\\ \tau_3\leq (\tau_1+\delta)\wedge t\wedge\tau^U}}\Pbf(R\geq\eps)\cv{\delta\to 0}0,
\]
where the supremum is taken on  $\tau_i$ stopping times and with
\[
R:=\left\{\begin{array}{ll}
d(X_{\tau_1},X_{\tau_2})\wedge d(X_{\tau_2},X_{\tau_3})&\text{if }0<\tau_1<\tau_2<\tau_3<\xi,\\
d(X_{\tau_2-},X_{\tau_2})\wedge d(X_{\tau_2},X_{\tau_3})&\text{if }0<\tau_1=\tau_2<\tau_3<\xi,\\
d(X_{\tau_1},X_{\tau_2})&\text{if }0=\tau_1\leq\tau_2<\xi\text{ or } 0<\tau_1<\tau_2<\tau_3=\xi,\\
d(X_{\tau_2-},X_{\tau_2})&\text{if }0<\tau_1=\tau_2<\tau_3=\xi,\\
0&\text{if }\tau_1=\xi\text{ or }0<\tau_1\leq\tau_2=\tau_3,\\
\infty&\text{if }0=\tau_1<\tau_2=\xi.
\end{array}\right.
\]
\end{enumerate}
\end{theorem}
\begin{remark}
1) If $d$ is obtained from a metric on $S^\Delta$, then if $\eps<d(\Delta,U)$ the expression of $\alpha(\eps,t,U,\delta)$ may be simplified as follows: 
\[
\alpha(\eps,t,U,\delta)=\sup_{\Pbf\in\Pcal}\sup_{\substack{\tau_1\leq\tau_2\leq\tau_3\\ \tau_3\leq (\tau_1+\delta)\wedge t\wedge\tau^U}}\Pbf(\widetilde{R}\geq\eps)\cv{\delta\to 0}0,
\]
where the supremum is taken on $\tau_i$ stopping times and with
\[
\widetilde{R}:=\left\{\begin{array}{ll}
d(X_{\tau_1},X_{\tau_2})\wedge d(X_{\tau_2},X_{\tau_3})&\text{if }0<\tau_1<\tau_2,\\
d(X_{\tau_2-},X_{\tau_2})\wedge d(X_{\tau_2},X_{\tau_3})&\text{if }0<\tau_1=\tau_2,\\
d(X_{\tau_1},X_{\tau_2})&\text{if }0=\tau_1.
\end{array}\right.
\]
2) It is straightforward to verify that a subset $D\subset \D(S)$ is relatively compact for $\D(S)$ if and only if $D$ is relatively compact for $\Dloc(S)$ and
\[
\forall t\geq 0,\quad\left\{x_s\mid x\in D,~s\leq t\right\}\Subset S.
\]
Hence we may recover the classical characterisation of compact sets of $\D(S)$ and the classical Aldous criterion. Moreover we may obtain a version of Theorem \ref{thmTension} for $\D(S)$.\\
3) The difficult part of Theorem \ref{thmTension} is the implication \ref{itemPropTen3}$\Rightarrow$\ref{itemPropTen2}, and its 
proof is adapted from the proof of Theorem 16.10 pp. 178-179 from \cite{Bi99}.
Roughly speaking the assertion \ref{itemPropTen3} uses
\[
\omega^{\prime\prime}_x(\delta):=\sup_{s_1\leq s_2\leq s_3\leq s_1+\delta}d(x_{s_1},x_{s_2})\wedge d(x_{s_2},x_{s_3}),
\]
while the Aldous criterion uses
\[
\omega_x(\delta):=\sup_{s_1\leq s_2\leq s_1+\delta}d(x_{s_1},x_{s_2}).
\]
The term $d(X_{\tau_2-},X_{\tau_2})$ appears because, in contrary to the deterministic case, some stopping time may not be approximate by the left. We refer to the proof of Theorem 12.4 pp. 132-133 from \cite{Bi99} for the relation between $\omega^{\prime\prime}$ and $\omega^\prime$.
\end{remark}
\begin{proof}[Proof of Theorem \ref{thmTension}]~

\noindent
\ref{itemPropTen2}$\Rightarrow$\ref{itemPropTen1}
Let $\eta>0$ be and consider $(t_n)_{n\geq 1}$ a sequence of times tending to infinity and $(K_n)_{n\geq 1}$ an increasing sequence of compact subsets of $S$ such that $S=\bigcup_nK_n$. For $n\geq 1$ define $\delta_n$ such that
\[
\sup_{\Pbf\in\Pcal}\Pbf\left(\omega^\prime_{t_n,K_n,X}(\delta_n)\geq n^{-1}\right)\leq 2^{-n}\eta\,.
\]
Set 
\[
D:=\{\forall n\in\N^*,~\omega^\prime_{t_n,K_n,X}(\delta_n)< n^{-1}\}.
\]
By Theorem \ref{thmComp},  $D$ is relatively compact and moreover
\[
\sup_{\Pbf\in\Pcal}\Pbf\left(D^c\right)\leq \sum_{n\geq 1}2^{-n}\eta =\eta\,,
\]
so $\Pcal$ is tight.\\
\ref{itemPropTen1}$\Rightarrow$\ref{itemPropTen3}
Let $\eps,\eta>0$ be arbitrary. There exists a compact set $D\subset\Dloc(S)$ such that
\[
\sup_{\Pbf\in\Pcal}\Pbf(D^c)\leq\eta.
\]
By Theorem \ref{thmComp}, there exists $\delta>0$ such that
\[
D\subset\{\omega^\prime_{t,\overline{U},X}(\delta)<\eps\}.
\]
Since for all $\tau_1\leq\tau_2\leq\tau_3\leq (\tau_1+\delta)\wedge t\wedge\tau^U$
we have 
\[
\{\omega^\prime_{t,\overline{U},X}(\delta)<\eps\}\subset\{R<\eps\},
\]
we conclude that
\[
\sup_{\Pbf\in\Pcal}\sup_{\substack{\tau_1\leq\tau_2\leq\tau_3\\ \tau_3\leq (\tau_1+\delta)\wedge t\wedge\tau^U}}\Pbf(R\geq\eps)\leq\eta\, .
\]
\ref{itemPropTen3}$\Rightarrow$\ref{itemPropTen2}
For all $\eps>0$, $t\geq 0$ and open subset $U\Subset S$, 
up to consider $\widetilde{\tau}_i:=\tau_i\wedge(\tau_1+\delta)\wedge t\wedge\tau^U$ we we have a new expression of $\alpha(\eps,t,U,\delta)$:
\begin{align}\label{eqThmTension2}
\alpha(\eps,t,U,\delta)=\sup_{\Pbf\in\Pcal}\sup_{\tau_1\leq\tau_2\leq\tau_3\leq\xi}\Pbf(R\geq\eps,~\tau_3\leq(\tau_1+\delta)\wedge t\wedge\tau^U)\cv{\delta\to 0}0.
\end{align}
Consider $\eps_0>0$, $t\geq 0$ and $K$ a compact subset of $S$. We need to prove that 
\[
\inf_{\Pbf\in\Pcal}\Pbf(\omega^\prime_{t,K,X}(\delta)<\eps_0)\cv{\delta\to 0}1.
\]
Choose  $0<\eps_1<\eps_0/4$  such that
\[
U:=\left\{y\in S\mid d(y,K)<\eps_1\right\}\Subset S.
\]
For $n\in\N$ and $\eps>0$, define inductively the stopping times 
(see Corollary \ref{corST})
\begin{align*}
&\tau_0:=0,\\
&\tau_n^\eps:=\inf\big\{s>\tau_n \big |\,d(X_{\tau_n},X_s)\vee d(X_{\tau_n},X_{s-})\geq\eps\big\}\wedge (t+2)\wedge\tau^U,\\
&\tau_{n+1}:=\tau_n^{\eps_1},
\end{align*}
It is clear that $\tau_n^\eps$ increases to $\tau_{n+1}$ when $\eps$ increases to $\eps_1$. If we choose $0<\eps_2<\eps_1$, then for all $\Pbf\in\Pcal$,
\begin{align}
\label{eqThmTension1}&\limsup_{\substack{\eps\to\eps_1\\ \eps<\eps_1}}\Pbf(X_{\tau_n}\in K,~\tau_n^\eps<\xi,~d(X_{\tau_n},X_{\tau_n^\eps})\leq\eps_2,~\tau_n^\eps\leq t+1)\\
\nonumber&\hspace{1cm}\leq \Pbf\left(\limsup_{\substack{\eps\to\eps_1\\ \eps<\eps_1}}\{X_{\tau_n}\in K,~\tau_n^\eps<\xi,~d(X_{\tau_n},X_{\tau_n^\eps})\leq\eps_2,~\tau_n^\eps\leq t+1\}\right)=\Pbf(\emptyset)=0.
\end{align}
For all $\Pbf\in\Pcal$, $\delta\leq 1$ and $0<\eps<\eps_1$ we have using the expression \eqref{eqThmTension2} with stopping times $0\leq\tau_0^\eps\leq\tau_0^\eps$
\begin{align*}
\Pbf(X_0\in K,~\tau_0^\eps\leq\delta)
& = \Pbf(X_0\in K,~X_{\tau_0^\eps}\not\in B(X_0,\eps_2),~\tau_0^\eps\leq\delta)\\
&\quad+\Pbf(X_0\in K,~\tau_0^\eps<\xi,~d(X_0,X_{\tau_0^\eps})< \eps_2,~\tau_0^\eps\leq\delta)\\
&\leq \alpha(\eps_2,t+2,U,\delta)\\
&\quad+\Pbf(X_0\in K,~\tau_0^\eps<\xi,~d(X_0,X_{\tau_0^\eps})< \eps_2,~\tau_0^\eps\leq t+1),
\end{align*}
so letting $\eps\to\eps_1$, since $\tau_0^\eps\uparrow\tau_1$, by \eqref{eqThmTension1} we obtain
\begin{align}\label{eqThmTension3}
\Pbf(X_0\in K,~\tau_1\leq\delta)\leq \alpha(\eps_2,t+2,U,\delta).
\end{align}
For all $\Pbf\in\Pcal$, $\delta\leq 1$, $n\in\N$ and $0<\eps<\eps_1$ we have also using the expression \eqref{eqThmTension2} with stopping times $\tau_n\leq\tau_n^\eps\leq\tau_{n+1}$ and $\tau_n\leq\tau_n^\eps\leq\tau_{n+1}^\eps$
\begin{align*}
\Pbf(\tau_{n+1}\leq t,~X_{\tau_n},X_{\tau_{n+1}}\in K,~\tau_{n+1}^\eps-\tau_n\leq\delta)\hspace{-7cm}&\hspace{14cm}\\
&\leq \Pbf(X_{\tau_n}\in K,~\tau_n^\eps<\xi,~d(X_{\tau_n},X_{\tau_n^\eps})<\eps_2,~\tau_n^\eps\leq t+1)\\
&\quad + \Pbf(X_{\tau_{n+1}}\in K,~\tau_{n+1}^\eps<\xi,~d(X_{\tau_{n+1}},X_{\tau_{n+1}^\eps})<\eps_2,~\tau_{n+1}^\eps\leq t+1)\\
&\quad +\Pbf\Big(\tau_{n+1}\leq t,~X_{\tau_n},X_{\tau_{n+1}}\in K,~d(X_{\tau_n},X_{\tau_n^\eps})\geq\eps_2,\\
&\pushright{d(X_{\tau_n^\eps},X_{\tau_{n+1}})\geq\frac{\eps_2}{2},\tau_{n+1}-\tau_n\leq\delta\Big)}\\
&\quad +\Pbf\Big(\tau_{n+1}\leq t,~X_{\tau_n},X_{\tau_{n+1}}\in K,~d(X_{\tau_n},X_{\tau_n^\eps})\geq\eps_2,~X_{\tau_{n+1}^\eps}\not\in B(X_{\tau_n^\eps},\frac{\eps_2}{2}),\\
&\pushright{\tau_{n+1}^\eps-\tau_n\leq\delta\Big)}\\
&\leq \Pbf(X_{\tau_n}\in K,~\tau_n^\eps<\xi,~d(X_{\tau_n},X_{\tau_n^\eps})\leq\eps_2,~\tau_n^\eps\leq t+1)\\
&\quad + \Pbf(X_{\tau_{n+1}}\in K,~\tau_{n+1}^\eps<\xi,~d(X_{\tau_{n+1}},X_{\tau_{n+1}^\eps})\leq\eps_2,~\tau_{n+1}^\eps\leq t+1)\\
&\quad +2\alpha\left(\frac{\eps_2}{2},t+2,U,\delta\right),
\end{align*}
so letting $\eps\to\eps_1$, since $\tau_{n+1}^\eps\uparrow\tau_{n+2}$, by \eqref{eqThmTension1} we obtain
\begin{align}\label{eqThmTension4}
\Pbf(\tau_{n+1}\leq t,~X_{\tau_n},X_{\tau_{n+1}}\in K,~\tau_{n+2}-\tau_n\leq\delta)\leq 2\alpha\left(\frac{\eps_2}{2},t+2,U,\delta\right).
\end{align}
For all $\Pbf\in\Pcal$, $\delta\leq 1$, $n\in\N^*$ and $0<\eps<\eps_1$ we can write using the expression \eqref{eqThmTension2} with stopping times $\tau_n\leq\tau_n\leq\tau_n^\eps$
\begin{align*}
\Pbf(\tau_n\leq t,~X_{\tau_n}\in K,~d(X_{\tau_n-},X_{\tau_n})\geq \eps_2,~\tau_{n}^\eps-\tau_n\leq\delta)\hspace{-7.5cm}&\\
& \leq \Pbf(X_{\tau_n}\in K,~\tau_n^\eps<\xi,~d(X_{\tau_n},X_{\tau_n^\eps})<\eps_2,~\tau_n^\eps\leq t+1)\\
&\quad+\Pbf(\tau_n\leq t,~X_{\tau_n}\in K,~d(X_{\tau_n-},X_{\tau_n})\geq \eps_2,~X_{\tau_n^\eps}\not\in B(X_{\tau_n},\eps_2),~\tau_{n}^\eps-\tau_n\leq\delta)\\
&\leq \Pbf(X_{\tau_n}\in K,~\tau_n^\eps<\xi,~d(X_{\tau_n},X_{\tau_n^\eps})\leq\eps_2,~\tau_n^\eps\leq t+1) +\alpha\left(\eps_2,t+2,U,\delta\right),
\end{align*}
so letting $\eps\to\eps_1$, since $\tau_n^\eps\uparrow\tau_{n+1}$, by \eqref{eqThmTension1} we obtain
\begin{align}\label{eqThmTension5}
\Pbf(\tau_n\leq t,~X_{\tau_n}\in K,~d(X_{\tau_n-},X_{\tau_n})\geq \eps_2,~\tau_{n+1}-\tau_n\leq\delta)\leq\alpha\left(\eps_2,t+2,U,\delta\right).
\end{align}
Let $m\in 2\N$ and $0<\delta^\prime\leq 1$ be such that $m>2t/\delta^\prime$ 
and denote the event
\[
A:=\left\{\tau_m\leq t\text{ and }\forall n\leq m,~X_{\tau_n}\in K\right\}.
\]
Then for all $0\leq i<m$, thanks to \eqref{eqThmTension4}
\begin{align*}
\Ebf\left[\tau_{i+2}-\tau_i\mid A\right]
&\geq \delta^\prime\Pbf\left(\tau_{i+2}-\tau_i\geq \delta^\prime\mid A\right)
\geq \delta^\prime\left(1-\frac{2\alpha\left(\frac{\eps_2}{2},t+2,U,\delta^\prime\right)}{\Pbf(A)}\right).
\end{align*}
Hence
\begin{align*}
t\geq \Ebf\left[\tau_m\mid A\right]=\sum_{i=0}^{(m-2)/2}\Ebf\left[\tau_{2i+2}-\tau_{2i}\mid A\right]\geq \frac{m\delta^\prime}{2}\left(1-\frac{2\alpha\left(\frac{\eps_2}{2},t+2,U,\delta^\prime\right)}{\Pbf(A)}\right),
\end{align*}
so
\begin{align}\label{eqThmTension6}
\Pbf(A)\leq \frac{2\alpha\left(\frac{\eps_2}{2},t+2,U,\delta^\prime\right)}{1-2t/(m\delta^\prime)}.
\end{align}
Taking  $0<\delta\leq 1$ and setting
\[
B_{m,\delta}:=\left\{\begin{array}{l}
(\tau_m,X_{\tau_0},\ldots,X_{\tau_m})\not\in[0,t]\times K^{m+1},\\
X_0\in K\Rightarrow\tau_1>\delta,\\
\forall 0\leq n\leq m-2,~\tau_{n+1}\leq t\text{ and }X_{\tau_n},X_{\tau_{n+1}}\in K\Rightarrow\tau_{n+2}-\tau_n>\delta,\\
\forall 0\leq n< m,~\tau_n\leq t,~X_{\tau_n}\in K,~d(X_{\tau_n-},X_{\tau_n})\geq \eps_2\Rightarrow\tau_{n+1}-\tau_n>\delta
\end{array}\right\},
\]
by \eqref{eqThmTension3}, \eqref{eqThmTension4}, \eqref{eqThmTension5} and \eqref{eqThmTension6} we obtain that 
\begin{align*}
\inf_{\Pbf\in\Pcal}\Pbf(B_{m,\delta})&\geq 1-\frac{2\alpha\left(\frac{\eps_2}{2},t+2,U,\delta^\prime\right)}{1-2t/(m\delta^\prime)}
-\alpha(\eps_2,t+2,U,\delta)\\
&\quad -2(m-1)\alpha\left(\frac{\eps_2}{2},t+2,U,\delta\right)
-m\alpha\left(\eps_2,t+2,U,\delta\right).
\end{align*}
Hence
\[
\sup_{m\in\N}\inf_{\Pbf\in\Pcal}\Pbf(B_{m,\delta})\cv{\delta\to 0}1.
\]
Recalling that $\eps_1< 4\eps_0$, a straightforward computation gives
\[
B_{m,\delta}\subset\left\{\omega^\prime_{t,K,X}(\delta)< \eps_0\right\}.
\]
We conclude that
\[
\inf_{\Pbf\in\Pcal}\Pbf(\omega^\prime_{t,K,X}(\delta)< \eps_0)\cv{\delta\to 0}1.
\]
\end{proof}

\section{Time change and Skorokhod topologies}
\subsection{Definition and properties of time change}

First we give the definition of the time change mapping (see also \S 6.1 pp. 306-311 from \cite{EK86}, \S V.26 pp. 175-177 from \cite{RW94}).

\begin{definition}[Time Change]
Let us introduce
\[
\rmC^{\not = 0}(S,\R_+):=\left\{g:S\to\R_+\mid \{g=0\} \text{ is closed and $g$ is continuous on }\{g\not =0\}\right\},
\]
and for $g\in\rmC^{\not = 0}(S,\R_+)$, $x\in\Dexp(S)$ and $t\in[0,\infty]$ we denote
\begin{align}\label{taug}
\tau_t^g(x):=\inf\left\{s\geq 0\mid A_s^g(x)\geq t\right\},\;\mbox{ where }\; A_t^g(x):=\left\{\begin{array}{ll}
\int_0^t\frac{\d u}{g(x_u)},&\text{if }t\in[0,\tau^{\{g\not=0\}}(x)],\\
\infty&\text{otherwise}.
\end{array}\right.
\end{align}
For $g\in\rmC^{\not = 0}(S,\R_+)$, we define a time change mapping, which is $\Fc$-measurable,
\[\begin{array}{cccc}
g\mycdot X:&\Dexp(S)&\rightarrow & \Dexp(S)\\
&x&\mapsto & g\mycdot x,
\end{array}\]
as follows:  for $t\in\R_+$
\begin{align}\label{eqdefTC}
(g\mycdot X)_t:=\left\{\begin{array}{ll}
X_{\tau_\infty^g-} & \text{if }t\geq A_{\tau_\infty^g}^g,~X_{\tau_\infty^g-}\text{ exists and belongs to }\{g=0\},\\
X_{\tau_t^g} & \text{otherwise}.
\end{array}\right.
\end{align}
For $g\in\rmC^{\not = 0}(S,\R_+)$ and $\Pbf\in\Pcal(\Dexp(S))$, we also define $g\mycdot\Pbf$ the pushforward of $\Pbf$ by $x\mapsto g\mycdot x$.
\end{definition}
The fact that this mapping is measurable will be proved in the next section.
\begin{remark}
Let us stress that, by using Corollary \ref{corST}, $\tau_t^g$ is a stopping time 
and, in particular
\[
\tau_\infty^g=\tau^{\{g\not=0\}}=\inf\left\{t\geq 0\mid g(X_{t-})\wedge g(X_t)=0\right\}\wedge\xi.
\]
The time of explosion of $g\mycdot X$ is given by 
\begin{equation*}
\xi(g\mycdot X) = \left\{\begin{array}{ll}
\infty & \text{if }\tau_\infty^g<\xi\text{ or }X_{\xi-}\text{ exists and belongs to }\{g=0\},\\
A_{\xi}^g & \text{otherwise}.
\end{array}\right.
\end{equation*}
\end{remark}
Roughly speaking $g\mycdot X$ is the solution of $(g\mycdot X)_t:=X_{\tau^g_t}$ with $\dot{\tau}^g_t:=g((g\mycdot X)_t)$, on the time interval  $[0,\tau_\infty^g)$. 

\begin{proposition}\label{propTC}~
\begin{enumerate}
\item For $U\subset S$ an open subset, by identifying
\[
\rmC(U,\R_+)=\big\{g\in\rmC^{\not = 0}(S,\R_+)\,|\,\{g\not=0\}\subset U\text{ and $g$ is continuous on }U\big\},
\]
the time change mapping 
\[\begin{array}{ccc}
\rmC(U,\R_+)\times\Dexp(S)&\rightarrow & \Dexp(S)\\
(g,x)&\mapsto & g\mycdot x,
\end{array}\]
is measurable between $\Bc(\rmC(U,\R_+))\otimes\Fc$ and $\Fc$.
\item If $g_1,g_2\in\rmC^{\not = 0}(S,\R_+)$ and $x\in \Dexp(S)$, then $g_1\mycdot(g_2\mycdot x)= (g_1g_2)\mycdot x$.
\item If $g$ is bounded and  belongs to $\rmC^{\not = 0}(S,\R_+)$, and  $x\in\D(S)$, then $g\mycdot x\in\D(S)$.
\item Define
\[
\widetilde{\rmC}^{\not = 0}(S,\R_+):=\big\{g\in \rmC^{\not = 0}(S,\R_+)\,|\,\forall K\subset S\text{ compact, }g(K)\text{ is bounded}\big\}.
\]
If $g\in\widetilde{\rmC}^{\not = 0}(S,\R_+)$ and $x\in\Dloc(S)$, then $g\mycdot x\in\Dloc(S)$.
\item If $g\in\rmC^{\not = 0}(S,\R_+)$ and $(\Pbf_a)_{a\in S}$ is a strong Markov family, then $(g\mycdot\Pbf_a)_{a\in S}$ is a Markov family. Furthermore, if $(\Pbf_a)_{a\in S}$ is a $\Fc_{t^+}$-strong Markov family, then $(g\mycdot\Pbf_a)_{a\in S}$ is a $\Fc_{t^+}$-strong Markov family.
\end{enumerate}
\end{proposition}
\begin{proof}
The first point is straightforward by using Proposition \ref{propMes}, while the second point 
is a direct consequence of the time change definition and, in particular, of the first part of \eqref{eqdefTC}.
The third point can be easily deduced because,
\[
\xi(g\mycdot x)\geq\int_0^{\infty}\frac{\d s}{g(x_s)}\geq \int_0^{\infty}\frac{\d s}{\|g\|}=\infty.
\]
To prove the fourth point we suppose that $\xi(g\mycdot x)<\infty$ and $\{g\mycdot x_s\}_{s<\xi(g\mycdot x)}\Subset S$. Then $\{x_s\}_{s<\xi(x)}=\{g\mycdot x_s\}_{s<\xi(g\mycdot x)}$ so
\[
\infty>\xi(g\mycdot x)=\int_0^{\xi(x)}\frac{\d s}{g(x_s)}\geq \frac{\xi(x)}{\|g\|_{\{x_s\}_{s<\xi(x)}}}.
\]
Hence $\xi(x)<\infty$ and so $g\mycdot x_{\xi(g\mycdot x)-}=x_{\xi(x)-}$ exists.
In proving the last point, we simplify the notation by  setting $\tau_t:=\tau_t^g$. It is straightforward (using especially Proposition \ref{propMes} and Corollary \ref{corST}) to obtain the following facts:

\noindent
\begin{itemize}
\item $\tau_t$ is a stopping time,
\item $g\mycdot X_t$ is $\Fc_{\tau_t}$-measurable,
\item $\left\{g\mycdot X_t\not = X_{\tau_t}\right\}\in\Fc_{\tau_t}$,
\item $g\mycdot X_t\not = X_{\tau_t}$ implies $g(g\mycdot X_t)=0$ and $(g\mycdot X_{t+s})_{s\geq 0}$ is constant,
\item $g\mycdot X_t = X_{\tau_t}$ imply $(g\mycdot X_{t+s})_{s\geq 0}=g\mycdot(X_{\tau_t+s})_{s\geq 0}$\,.
\end{itemize}
Suppose that $(\Pbf_a)_{a\in S}$ is a $(\Fc_t)_t$-strong Markov family, then for any $t_0\in\R_+$, $a\in S^\Delta$ and $B\in\Fc$, $\Pbf_a$-a.s.
\begin{align*}
\Pbf_a\left((g\mycdot X_{t_0+t})_t\in B\mid\Fc_{\tau_{t_0}}\right)
& = \Pbf_a\left((g\mycdot X_{t_0+t})_t\in B,~g\mycdot X_{t_0}= X_{\tau_{t_0}}\mid\Fc_{\tau_{t_0}}\right)\\
&\hspace{.5cm} + \Pbf_a\left((g\mycdot X_{t_0+t})_t\in B,~g\mycdot X_{t_0}\not= X_{\tau_{t_0}}\mid\Fc_{\tau_{t_0}}\right) \\
& = \Pbf_a\left(g\mycdot (X_{\tau_{t_0}+t})_t\in B\mid\Fc_{\tau_{t_0}}\right)\1_{\{g\mycdot X_{t_0}= X_{\tau_{t_0}}\}}\\
&\hspace{.5cm} + \Pbf_a\left((g\mycdot X_{t_0})_t\in B\mid\Fc_{\tau_{t_0}}\right)\1_{\{g\mycdot X_{t_0}\not= X_{\tau_{t_0}}\}} \\
&=\Pbf_{X_{\tau_0}}(g\mycdot X\in B)
=g\mycdot \Pbf_{g\mycdot X_{t_0}}(B).
\end{align*}
Hence $(g\mycdot P_a)_{a\in S}$ is a $(\Fc_t)_t$-Markov family.
If $(\Pbf_a)_a$ is a $\Fc_{t^+}$-strong Markov family, then for all $(\Fc_{\tau_t^+})_t$-stopping time $\sigma$,
\[
\{\tau_\sigma<t\} = \bigcup_{q\in\Q_+}\{\sigma<q,~\tau_q<t\}\in\Fc_t,
\]
so $\tau_\sigma$ is a $\Fc_{t^+}$-stopping time. Using the same argument as before we obtain that $(g\mycdot\Pbf_a)_a$ is a $\Fc_{t^+}$-strong Markov family.
\end{proof}
Another interesting fact is the following:
\begin{theorem}[Continuity of the time change]\label{thmContTC}
Denote by $B$ the set of couples $(g,x)\in\widetilde{\rmC}^{\not = 0}(S,\R_+)\times\Dloc(S)$ such that
\begin{align}\label{eqB1}
&\tau_\infty^g(x)<\xi(x)\;\mbox{implies }\;\int_0^{\tau_\infty^g(x)+}\frac{\d s}{g(x_s)}=\infty,
\end{align}
and such that 
\begin{align}\label{eqB2}
&A^g_{\tau_\infty^g(x)}(x)<\infty,~x_{\tau_\infty^g(x)-}\text{ exists in }S\text{ and }g(x_{\tau_\infty^g(x)-})=0\;\mbox{ imply }\; x_{\tau_\infty^g(x)-}=x_{\tau_\infty^g(x)}.
\end{align}
Then the time change
\[\begin{array}{ccc}
\widetilde{\rmC}^{\not = 0}(S,\R_+)\times\Dloc(S)&\rightarrow & \Dloc(S)\\
(g,x)&\mapsto & g\mycdot x
\end{array}\]
is continuous on $B$ when we endow respectively $\widetilde{\rmC}^{\not = 0}(S,\R_+)$ with the 
topology of uniform convergence on compact sets and $\Dloc(S)$ with the local Skorokhod topology. In particular
\[\begin{array}{ccc}
\rmC(S,\R_+^*)\times\Dloc(S)&\rightarrow & \Dloc(S)\\
(g,x)&\mapsto & g\mycdot x
\end{array}\]
is continuous for the topologies  of uniform convergence on compact sets and local Skorokhod topology.
\end{theorem}
\begin{remark}\label{rkContTC}
1) It is not difficult to prove that $B$ is the continuity set.\\
2) If $(g,x)\in B$ and $h\in \widetilde{\rmC}^{\not = 0}(S,\R_+)$ is such that $\{h=0\}=\{g=0\}$ and $h\leq Cg$ for a constant $C\in\R_+$, then $(h,x)\in B$.\\
3) More generally, let $B_0$ be the set of $(g,x)\in\widetilde{\rmC}^{\not = 0}(S,\R_+)\times\Dloc(S)$ such that
\begin{align*}
&\tau_\infty^g(x)<\infty\Rightarrow \forall t\geq 0,~x_{\tau_\infty^g(x)+t}=x_{\tau_\infty^g(x)},\\
&x_{\tau_\infty^g(x)-}\text{ exists in }S\text{ and }g(x_{\tau_\infty^g(x)-})=0\Rightarrow x_{\tau_\infty^g(x)-}=x_{\tau_\infty^g(x)}.
\end{align*}
Then
\begin{align}\label{eqRkContTC}
\left\{(g,g\mycdot x)\mid (g,x)\in\widetilde{\rmC}^{\not = 0}(S,\R_+)\times\Dloc(S)\right\}\subset B_0\subset B.
\end{align}
4) A similar theorem may be proved for $g\mycdot X:\rmC^{\not = 0}(S,\R_+)\times\Dexp(S)\to \Dexp(S)$ using the topology described in Remark \ref{rkOthLocSko}.
\end{remark}

To simplify the proof of the theorem we use a technical result containing a construction of a sequence of bi-Lipschitz bijections $(\lambda^k)_{k}$ useful when proving the convergence. Before stating this result let us note that, for any  $(g,x)\in\widetilde{\rmC}^{\not = 0}(S,\R_+)\times\Dloc(S)$ and any $t\leq\tau_\infty^g(x)$ such that $\{x_s\}_{s<t}\Subset\{g\not = 0\}$,  by using \eqref{taug}, $A^g_t(x)<\infty$.
\begin{lemma}\label{lmContTC}
Take a metric $d$ of $S$. Let $x,x^k\in \Dloc(S)$ and $g,g_k\in\widetilde{\rmC}^{\not = 0}(S,\R_+)$ be such that $(g_k,x^k)$ converges to $(g,x)$, as $k\to\infty$. Let $t\leq\tau_\infty^g(x)$ be such that $\{x_s\}_{s<t}\Subset\{g\not = 0\}$. Then
\begin{itemize}
\item[i)] there exists a sequence $(\lambda^k)_k\in\Lambda^\N$ such that, for $k$ large enough $\lambda^k_{A^g_t(x)}\leq\xi(g_k\mycdot x^k)$ and
\[
\|\log\dot{\lambda}^k\|\to 0,\quad \sup_{v<A^g_t(x)}d(g\mycdot x_v,g_k\mycdot x^k_{\lambda^k_v})\to 0,\quad g_k\mycdot x^k_{\lambda^k_{A^g_t(x)}}\to g\mycdot x_{A^g_t(x)},\;\mbox{ as }k\to\infty.
\]
\item[ii)] Moreover, if $\tau_\infty^g(x)<\xi(x)$ and $\int_0^{\tau_\infty^g(x)+}\frac{\d s}{g(x_s)}=\infty$, $(\lambda^k)_k$ may be chosen such that for any $v\geq 0$ 
and $k$ large enough
$\lambda^k_{A^g_t(x)+v}<\xi(g_k\mycdot x^k)$ and
\[
\limsup_{k\to\infty}\sup_{A^g_t(x)\leq w\leq A^g_t(x)+v}d\left(g_k\mycdot x^k_{\lambda^k_w},\{x_s\}_{t\leq s\leq \tau_\infty^g(x)}\right)=0.
\]
\end{itemize}
\end{lemma}

We postpone the proof of the lemma and we give the proof of the continuity of time change:
\begin{proof}[Proof of Theorem \ref{thmContTC}]
We remark first that
\[
B=B_1\cup B_2\cup B_3\cup B_4,
\]
with
\begin{align*}
&B_1:=\left\{A^g_{\tau_\infty^g(x)}(x)=\infty\text{ or }\{x_s\}_{s<\tau_\infty^g(x)}\not\Subset S\right\},\\
&B_2:=\left\{\tau_\infty^g(x)=\xi(x)<\infty\text{ and }\{x_s\}_{s<\tau_\infty^g(x)}\Subset\{g\not = 0\}\right\},\\
&B_3:=\left\{\tau_\infty^g(x)<\xi(x),~x_{\tau_\infty^g(x)-}=x_{\tau_\infty^g(x)},~A^g_{\tau_\infty^g(x)}(x)<\infty\text{ and }\int_0^{\tau_\infty^g(x)^+}\frac{\d s}{g(x_s)}= \infty\right\},\\
&B_4:=\left\{\tau_\infty^g(x)<\xi(x),~g(x_{\tau_\infty^g(x)-})\not=0\text{ and }\int_0^{\tau_\infty^g(x)^+}\frac{\d s}{g(x_s)}= \infty\right\}.
\end{align*}
Let $x,x^k\in \Dloc(S)$ and $g,g_k\in\widetilde{\rmC}^{\not = 0}(S,\R_+)$ be such that $(g_k,x^k)$ converge to $(g,x)$ and $(g,x)\in B$.
We need to prove that
\begin{equation}\label{eqContTC}
g_k\mycdot x^k\xrightarrow[k\to\infty]{\Dloc(S)}g\mycdot x,
\end{equation}
and we will decompose the proof with respect to values of $i$ such that  $(g,x)\in B_i$.
\begin{itemize}
\item If $(g,x)\in B_1$, we use the first part of Lemma \ref{lmContTC} for all $t<\tau_\infty^g(x)$.  We obtain that $A^g_t(x)<\xi(g\mycdot x)$. Since
$A^g_t(x)$ tends to $\xi(g\mycdot x)$, when $t$ tends to $\tau_\infty^g(x)$, by a diagonal extraction procedure we deduce \eqref{eqContTC}.
\item If $(g,x)\in B_2$, it suffices to apply the first part of Lemma \ref{lmContTC} to $t:=\xi(x)$ and $A^g_t(x)=\xi(g\mycdot x)$.
\item If $(g,x)\in B_3$, let $t<\tau_\infty^g(x)$ be. Then, by the third part of Lemma \ref{lmContTC} there exists $\lambda^k\in\Lambda$ such that, for any $v\geq 0$, for $k$ large enough, $\lambda^k_{A^g_t(x)+v}<\xi(g_k\mycdot x^k)$  and
\[
\|\log\dot{\lambda}^k\|\cv{k\to \infty}0,\quad
\limsup_{k\to\infty}\sup_{w\leq A^g_t(x)+v}d(g\mycdot x_w,g_k\mycdot x^k_{\lambda^k_w})\leq 2d\left(x_{\tau_\infty^g(x)},\{x_s\}_{t\leq s\leq \tau_\infty^g(x)}\right).
\]
Since $x$ is continuous at $\tau_\infty^g(x)$, we conclude by a diagonal extraction procedure, by letting $t$ tends to $\tau_\infty^g(x)$ and $v\to\infty$.
\item If $(g,x)\in B_4$, let $t=\tau_\infty^g(x)$ be. By the second part of Lemma \ref{lmContTC} there exists $\lambda^k\in\Lambda$ such that, for any $v\geq 0$, for  $k$ large enough $\lambda^k_{A^g_t(x)+v}<\xi(g_k\mycdot x^k)$, and
\[
\|\log\dot{\lambda}^k\|\cv{k\to \infty}0,\quad \sup_{w\leq A^g_t(x)+v}d(g\mycdot x_w,g_k\mycdot x^k_{\lambda^k_w})\cv{k\to\infty}0.
\]
We conclude by a diagonal extraction procedure and letting $v\to\infty$.
\end{itemize}
\end{proof}

\begin{proof}[Proof of Lemma \ref{lmContTC}]
Let $\widetilde{\lambda}^k\in\Lambda$ be as in Theorem \ref{thmLocSkoTop} and to simplify notations define, for $s\geq 0$
\begin{align*}
&\tau_s:=\tau^g_s(x),&A_s:=A_s^g(x),\\
&\tau^k_s:=\tau^{g_k}_s(x^k),&A^k_s:=A^{g_k}_s(x^k),
\end{align*}
and $u:=A_t$. Since $\tau_u=t\leq \xi(x)$ and $\{x_s\}_{s<t}\Subset S$
we have, for $k$ large enough $\widetilde{\lambda}^k_t\leq \xi(x^k)$, and  $\|\log\dot{\widetilde{\lambda}}^k\|_t\cv{}0$,  $\sup_{s<t}d(x_s,x^k_{\widetilde{\lambda}^k_s})\cv{}0$ and $x^k_{\widetilde{\lambda}^k_t}\cv{}x_t$, as $k\to\infty$.
Since $\{x_s\}_{s<t}\Subset\{g\not =0\}$, we deduce that for $k$ large enough $\{x^k_s\}_{s< \widetilde{\lambda}^k_t}\Subset\{g_k\not =0\}$.
Define then $\lambda^k\in\Lambda$ by
\[\left\{\begin{array}{ll}
\displaystyle \lambda^k_v:=A^k_{\widetilde{\lambda}^k_{\tau_v}}=\int_0^v\frac{g(x_{\tau_w})}{g_k(x^k_{\widetilde{\lambda}^k_{\tau_w}})}\dot{\widetilde{\lambda}}^k_{\tau_w}\d w & \text{if }v\leq u,\\
\displaystyle \dot{\lambda}^k_v =1 & \text{if }v> u.\\
\end{array}\right.\]
Since $\widetilde{\lambda}^k_t\leq\tau^k_\infty$ we have
\[
\lambda^k_u\leq A^k_{\tau^k_\infty}\leq\xi(g_k\mycdot x^k),
\]
now we obtain
\begin{align*}
&\sup_{v< u}d(g\mycdot x_{v},g_k\mycdot x^k_{\lambda^k_v})=\sup_{v<u}d(x_{\tau_v},x^k_{\widetilde{\lambda}^k_{\tau_v}})=\sup_{s<t}d(x_s,x^k_{\widetilde{\lambda}^k_s})\cv{k\to \infty}0,\\
&g_k\mycdot x^k_{\lambda^k_u}=x^k_{\widetilde{\lambda}^k_t}\cv{k\to\infty} x_t=g\mycdot x_u\\
&\|\log\dot{\lambda}^k\|=\esssup_{v\leq u}\left|\log\frac{\dot{\widetilde{\lambda}}^k_{\tau_v}g(x_{\tau_v})}{g_k(x^k_{\widetilde{\lambda}^k_{\tau_v}})}\right|=\esssup_{s\leq \tau_u}\left|\log\frac{\dot{\widetilde{\lambda}}^k_sg(x_s)}{g_k(x^k_{\widetilde{\lambda}^k_s})}\right|\cv{k\to \infty}0.
\end{align*}
For the second part of the proposition we keep the same construction as previously. 
For any $v\geq 0$ we have that
\[
\tau^k_{\lambda^k_{u+v}}=\inf\left\{t\geq \lambda^k_u\mid \int_{\lambda^k_u}^t\frac{\d s}{g_k(x^k_s)}\geq v\right\}\wedge\tau^k_\infty.
\]
Using Fatou's lemma
\[
\liminf_{k\to\infty}\int_{\lambda^k_u}^{(\widetilde{\lambda}^k_{\tau_\infty})^+}\frac{\d s}{g_k(x^k_s)}
= \liminf_{k\to\infty}\int_t^{\tau_\infty+}\frac{\dot{\widetilde{\lambda}}^k_{s}\d s}{g_k(x^k_{\lambda^k_s})}
\geq \int_t^{\tau_\infty+}\frac{\d s}{g(x_s)} = \infty,
\]
so $\limsup_{k\to\infty}\tau^k_{\lambda^k_{u+v}}-\widetilde{\lambda}^k_{\tau_\infty}\leq 0$.
Moreover, for $k$ large enough, $\tau^k_{\lambda^k_{u+v}}\geq\tau^k_{\lambda^k_u}=\widetilde{\lambda}^k_t$, so $\lambda^k_{u+v}<\xi(g_k\mycdot x^k)$  and
\[
\limsup_{k\to\infty}\sup_{u\leq w\leq u+v}d\left(g_k\mycdot x^k_{\lambda^k_w},\{x_s\}_{t\leq s\leq \tau_\infty}\right)=0.
\]
\end{proof}

\subsection{Connection between local and global Skorokhod topologies}
Generally to take into account the explosion, one considers processes in $\D(S^\Delta)$,
the set of \cadlag processes described in Definition \ref{defCADLAG}, associated to the space $S^\Delta$, and endowed with the global Skorokhod topology (see Corollary \ref{corGlobSkoTop}).
More precisely, the set of \cadlag paths with values in $S^\Delta$ is given by 
\begin{equation*}
\D(S^\Delta)=\left\{x\in(S^\Delta)^{\R_+}\mid
\begin{array}{l}\forall t\geq 0,\;x_t=\lim_{s\downarrow t}x_s,\;\mbox{ and }\; \\
\forall t>0,\;x_{t-}:=\lim_{s\uparrow t}x_s~\text{ exists in  }S^\Delta\end{array}\right\}\,.
\end{equation*}
A sequence $(x^k)_k$ in $\D(S^\Delta)$ converges to $x$ for the global Skorokhod topology if and only if there exists a sequence $(\lambda^k)_k$ of increasing homeomorphisms on $\R_+$ such that
\[
\forall t\geq 0,\quad\lim_{k\to\infty}\sup_{s\leq t}d(x_s,x^k_{\lambda^k_s})=0
\quad
\mbox{and}\quad
\lim_{k\to\infty}\|\lambda^k -{\rm id}\|_t=0.
\]
The  global Skorokhod topology does not depend on the arbitrary  metric $d$ on $S^\Delta$, but only on the topology on $S$.

In this section we give the connection between $\D(S^\Delta)$ with the global Skorokhod topology and $\Dloc(S)$ with the local Skorokhod topology.

We first identify these two measurable subspaces
\begin{align*}
\Dloc(S)\cap\D(S^\Delta)
&=\left\{x\in\Dloc(S)\mid0<\xi(x)<\infty\Rightarrow x_{\xi(x)-}\text{ exist in }S^\Delta\right\}\\
&=\left\{x\in\D(S^\Delta)\mid\forall t\geq\tau^{S},~x_t=\Delta\right\}.
\end{align*}
We can summarise our trajectories spaces by
\[\begin{array}{ccc}
\D(S)\subset &\Dloc(S)\cap\D(S^\Delta)&\subset\Dloc(S)\subset\Dexp(S).\\
&\cap\\
&\D(S^\Delta)
\end{array}\]
Hence $\Dloc(S)\cap\D(S^\Delta)$ will be endowed with two topologies, the local topology from $\Dloc(S)$ and the global topology from $\D(S^\Delta)$. 
\begin{remark}\label{rkGloImpLoc}
1) On $\Dloc(S)\cap\D(S^\Delta)$ the trace topology from $\Dloc(S)$ is weaker than the trace topology from $\D(S^\Delta)$. Eventually, these two topologies coincide on $\D(S)$.
Indeed this is clear using a metric $d$ on $S^\Delta$ and the characterisations of topologies given in Theorem \ref{thmLocSkoTop} and Corollary \ref{corGlobSkoTop}. The result 
in Corollary \ref{corConnection} below  is a converse sentence of the present remark.\\
2) If $x\in \Dloc(S)\cap\D(S^\Delta)$ then $g\mycdot x$ is well defined in $\Dloc(S)\cap\D(S^\Delta)$ for  
\[
g\in \rmC_b(S,\R_+^*)\subset\big\{g\in\rmC^{\not = 0}(S^\Delta,\R_+)\,|\, g(\Delta)=0\big\}.
\]
We deduce from Theorem \ref{thmContTC} and the third point of Remark \ref{rkContTC} that the mapping
\[\begin{array}{ccc}
\rmC_b(S,\R_+^*)\times\Dloc(S)\cap\D(S^\Delta)&\rightarrow & \Dloc(S)\cap\D(S^\Delta)\\
(g,x)&\mapsto & g\mycdot x
\end{array}\]
is continuous between the topology of the uniform convergence and the global Skorokhod topology.
\end{remark}

The following result is stated in a very general form because it will be useful when 
studying, for instance,  the martingale problems.

\begin{proposition}[Connection  between $\Dloc(S)$ and $\D(S^\Delta)$]\label{propEqLocGlo}
Let $\widetilde{S}$ be an arbitrary locally compact Hausdorff space with countable base and consider
\[\begin{array}{cccc}
\Pbf:&\widetilde{S}&\to&\Pcal(\Dloc(S))\\
&a&\mapsto&\Pbf_a
\end{array}\]
a weakly continuous mapping for the local Skorokhod topology. Then for any open subset $U$ of $S$, there exists $g\in\rmC(S,\R_+)$ such that $\{g\not =0\}=U$, for all $a\in\widetilde{S}$
\[
g\mycdot\Pbf_a\left(0<\xi<\infty\Rightarrow X_{\xi-}\text{ exists in }U\right)=1,
\]
and the application
\[\begin{array}{cccc}
g\mycdot\Pbf:&\widetilde{S}&\to&\Pcal(\{0<\xi<\infty\Rightarrow X_{\xi-}\text{ exists in }U\})\\
&a&\mapsto&g\mycdot\Pbf_a
\end{array}\]
is weakly continuous for the global Skorokhod topology from $\D(S^\Delta)$.
\end{proposition}
Before giving the proof of Proposition \ref{propEqLocGlo} we point out a direct application: we take $\widetilde{S}:=\N\cup\{\infty\}$, $U=S$ and a sequence of Dirac probability measures $\Pbf_k=\delta_{x^k}$, $\Pbf_\infty=\delta_{x}$. Then we deduce from Proposition \ref{propEqLocGlo} the following:
\begin{corollary}[Another description of $\Dloc(S)$]\label{corConnection}
Let $x,x^1,x^2,\ldots\in\Dloc(S)$ be. Then the sequence $x^k$ converges to $x$ in $\Dloc(S)$, as $k\to\infty$, if and only if there exists $g\in\rmC(S,\R_+^*)$ such that $g\mycdot x, g\mycdot x^1, g\mycdot x^2,\ldots\in\D(S^\Delta)$, and $g\mycdot x^k$ converges to $g\mycdot x$ in $\D(S^\Delta)$, as $k\to\infty$.
\end{corollary}
We proceed with the proof of Proposition \ref{propEqLocGlo} and, firstly we state a important result which will be our main tool:
\begin{lemma}\label{lmEqLocGlo}
Let $D$ be a compact subset of $\Dloc(S)$ and $U$ be an open subset of $S$. There exists $g\in\rmC(S,\R_+)$ such that:
\begin{enumerate}
\item[i)] $\{g\not =0\}=U$,
\item[ii)] for all $x\in D$, $(g,x)$ is in the set $B$ given by \eqref{eqB1}-\eqref{eqB2} in Theorem \ref{thmContTC} and 
\[g\mycdot x\in\big\{0<\xi<\infty\Rightarrow X_{\xi-}\text{ exists in }U\big\}.\]
\item[iii)] the trace topologies of $\Dloc(S)$ and $\D(S^\Delta)$ coincide on $\{g\mycdot x \,|\,x\in D\}$.
\end{enumerate}
Furthemore, if $g\in\rmC(S,\R_+)$ satisfies i)-iii) and if  $h\in\rmC(S,\R_+)$ is such that $\{h\not =0\}=U$ and $h\leq Cg$ with a non-negative constant $C$, then $h$ also satisfies
i)-iii).
\end{lemma}

\begin{proof}[Proof of Proposition \ref{propEqLocGlo}]
Let $(\widetilde{K}_n)_{n\in\N^*}$ be an increasing sequence of compact subset of $\widetilde{S}$ such that $\widetilde{S}=\bigcup_n\widetilde{K}_n$, then $\{\Pbf_a\}_{a\in \widetilde{K}_n}$ is tight, for all $n\in\N^*$. So, there exist subsets $D_n\subset\Dloc(S)$ which are compacts for the topology of $\Dloc(S)$, and such that
\[
\sup_{a\in \widetilde{K}_n}\Pbf_a(D_n^c)\leq\frac{1}{n}.
\]
For any $n\in\N^*$, let $g_n$ be satisfying i)-iii) of Lemma \ref{lmEqLocGlo} associated to the
compact set $D_n$. It is no difficult to see that there exists $g\in\rmC(S,\R_+)$ such that $\{g\not =0\}=U$ and for all $n\in\N^*$, $g\leq C_ng_n$ for non-negative constants  $C_n$.  Hence $g$ satisfies i)-iii) for all $D_n$, $n\in\N^*$. Hence for all $a\in\widetilde{S}$
\[
g\mycdot\Pbf_a\left(0<\xi<\infty\Rightarrow X_{\xi-}\text{ exists in }U\right)\geq\Pbf_a\Big(\bigcup_{n\in\N^*}D_n\Big)=1.
\]
Let $a_k,a\in\widetilde{S}$ such that $a_k\cv{k\to\infty}a$. For $n$ large enough $\{a_k\}_k\subset \widetilde{K}_n$. Then if $F$ is a subset of $\left\{0<\xi<\infty\Rightarrow X_{\xi-}\text{ exists in }U\right\}$ which is closed  for the topology of $\D(S^\Delta)$, then
\begin{align*}
&\limsup_{k\to\infty}g\mycdot\Pbf_{a_k}(F)-g\mycdot\Pbf_a(F)\\
&\hspace{2cm}\leq\limsup_{k\to\infty}\Pbf_{a_k}(X\in D_n,~g\mycdot X\in F)-\Pbf_a(X\in D_n,~g\mycdot X\in F)+\frac{1}{n}.
\end{align*}
But thanks to iii) in Lemma \ref{lmEqLocGlo}, $\left\{X\in D_n,~g\mycdot X\in F\right\}$ is a  subset of $\Dloc(S)$ which is closed  for the topology of $\Dloc(S)$. Hence by using the Portmanteau theorem (see for instance Theorem 2.1 from \cite{Bi99}, p. 16)
\[
\limsup_{k\to\infty}\Pbf_{a_k}(X\in D_n,~g\mycdot X\in F)\leq\Pbf_a(X\in D_n,~g\mycdot X\in F)
\] 
and so letting $n\to\infty$
\[
\limsup_{k\to\infty}g\mycdot\Pbf_{a_k}(F)\leq g\mycdot\Pbf_a(F).
\]
By using the Portmanteau theorem, the proof of the proposition is complete, except for the proof of Lemma \ref{lmEqLocGlo}.
\end{proof}

\begin{proof}[Proof of Lemma \ref{lmEqLocGlo}]
Let $d$ be a metric on $S^\Delta$ and denote 
\[K_n:=\left\{a\in S\mid d(a,S^\Delta\backslash U)\geq2^{-n}\right\}.\]
By using Theorem \ref{thmComp},there exists a sequence $(\eta_n)_n\in(0,1)^\N$ decreasing to $0$ such that
\begin{align}\label{eqLmEqLocGlo}
\sup_{x\in D}\omega^\prime_{2^n,B(\Delta,2^{-n-2})^c,x}(\eta_n)<2^{-n-2}.
\end{align}
Moreover, there exists $g\in\rmC(S^\Delta,[0,1])$ such that $\{g\not=0\}=U$ and
$g_{|K_n^c}\leq2^{-n}\eta_{n}$.\\
Let $x\in D$ be. We consider the following two situations:
\begin{itemize}
\item If $\tau^g_\infty(x)<\infty$ and $\{x_s\}_{s<\tau^g_\infty(x)}$ is not a compact of $U$, take $m\in \N$  such that $2^{m}\geq\tau^g_\infty(x)$, denote
\[
t:=\min\left\{s\geq 0\mid x_s\not\in\overset{\circ}{K}_{m+1}\right\}<\tau^g_\infty(x)
\]
and let $n\geq m$ be such that $x_t\in K_{n+2}\backslash \overset{\circ}{K}_{n+1}$. Using \eqref{eqLmEqLocGlo} there exist $t_1,t_2\in\R_+$ such that $t_1\leq t<t_2<\tau^g_\infty(x)$, $t_2-t_1>\eta_n$ and $x_s\not\in K_n$ for all $s\in[t_1,t_2)$. So 
\[
A^g_{\tau^g_\infty(x)}(x)\geq\int_{t_1}^{t_2}\frac{\d s}{g(x_s)}\geq 2^{m},
\] 
hence letting $m$ goes to infinity
\[
A^g_{\tau^g_\infty(x)}(x)=\infty.
\]
\item If $\tau^g_\infty(x)<\xi(x)$ and $g(x_{\tau^g_\infty(x)-})\not=0$, then $g(x_{\tau^g_\infty(x)})=0$.  Let $m\in \N$ be such that $2^{m}\geq\tau^g_\infty(x)$ and $\{x_s\}_{s\leq\tau^g_\infty(x)}\subset B(\Delta,2^{-m-2})^c$. Using \eqref{eqLmEqLocGlo}, there exist $t_1,t_2\in\R_+$ such that $t_1\leq \tau^g_\infty(x)<t_2<\xi(x)$, $t_2-t_1>\eta_m$ and $x_s\not\in K_{m}$ for all $s\in[t_1,t_2)$. So 
\[
\int_{0}^{\tau^g_\infty(x)+\eta_{m}}\frac{\d s}{g(x_s)}\geq 
\int_{t_1}^{t_1+\eta_{m}}\frac{\d s}{g(x_s)}\geq 
2^{m}
\]
 hence letting $m$ tend to infinity
\[
\int_{0}^{\tau^g_\infty(x)+}\frac{\d s}{g(x_s)}=\infty.
\]
\end{itemize}
Hence we obtain that $(g,x)\in B$ and $g\mycdot x\in\{0<\xi<\infty\Rightarrow X_{\xi-}\text{ exists in }U\}$ and ii) is verified.

We proceed by proving iii). Thanks to Remark \ref{rkGloImpLoc}, to get the equivalence of the topologies it is enough to prove that if $x^k,x\in D$ are such that $g\mycdot x^k\to g\mycdot x$ for the topology from $\Dloc(S)$ and $\xi(g\mycdot x)<\infty$, then the convergence also holds for the topology from $\D(S^\Delta)$. Let $\lambda^k\in\Lambda$ be such that
\[
\sup_{s\leq\xi(g\mycdot x)}d(g\mycdot x_s,g\mycdot x^k_{\lambda^k_s})\cv{}0,\hspace{1cm}\|\log\dot{\lambda}^k\|_{\xi(g\mycdot x)}\cv{}0,\quad
\mbox{ as }k\to\infty.
\]
We may suppose that $\dot{\lambda}^k_s=0$, for $s\geq \xi(g\mycdot x)$. Denote $t_k:=\lambda^k_{\xi(g\mycdot x)}$ and choose $m\in\N$ be such that $\{g\mycdot x_s\}_{s<\xi(g\mycdot x)}\Subset\overset{\circ}{K}_{m}$ and $\xi(g\mycdot x)<2^{m}$. Then, for $k$ large enough $\{g\mycdot x^k_s\}_{s<t_k}\Subset\overset{\circ}{K}_{m}$, $g\mycdot x^k_{t_k}\not \in K_{m+1}$ and $t_k<2^{m}$.
\begin{itemize}
\item Either $g\mycdot x^k_{t_k}\not\in U$ and so $g\mycdot x^k_{\lambda^k_s}=g\mycdot x^k_{t^k}$ for all $s\geq \xi(g\mycdot x)$.
\item Or $g\mycdot x^k_{t_k}\in U$ and let $n\geq m$ be such that $g\mycdot x^k_{t_k}=x_{\tau^g_{t_k}(x^k)}\in K_{n+2}\backslash\overset{\circ}{K}_{n+1}$. Using \eqref{eqLmEqLocGlo}, $d(x_s,x^k_{t_k})<2^{-n-2}$ and so $x_s\in U\backslash K_n$ for all $s\in[\tau^g_{t_k}(x^k),\tau^g_{t_k}(x^k)+\eta_n]$.
Hence $A^g_{\tau^g_{t_k}(x^k)+\eta_n}\geq t_k+2^n$, so $d(g\mycdot x_s,g\mycdot x^k_{t_k})<2^{-n-2}$ for all $s\in[t_k,t_k+2^{n}]$.
\end{itemize}
Hence we obtain that for $k$ large enough
\[
\sup_{s\leq\xi(g\mycdot x)+2^{m}}d(g\mycdot x_s,g\mycdot x^k_{\lambda^k_s})\leq \sup_{s\leq\xi(g\mycdot x)}d(g\mycdot x_s,g\mycdot x^k_{\lambda^k_s}) +2^{-m-2},
\]
so letting $m$ goes to the infinity we obtain that $g\mycdot x^k$ converge to $g\mycdot x$ for the global Skorokhod topology from $\D(S^\Delta)$. Hence the proof of iii) is done.

Finally, to prove the last part of the lemma let $g\in\rmC(S,\R_+)$ be such that  i)-iii) are satisfied and let $h\in\rmC(S,\R_+)$ be such that $\{h\not =0\}=U$ and $h\leq Cg$ with a non-negative constant $C$. Thanks to Remark \ref{rkContTC}, $(h,x)$ belongs to the set $B$ given by \eqref{eqB1}-\eqref{eqB2}, it is also clear that $h\mycdot x\in \{0<\xi<\infty\Rightarrow X_{\xi-}\text{ exists in }U\}$. We have that $\frac{h}{g}\in\rmC_b(U,\R_+^*)$, so using \eqref{eqRkContTC} for $S$ and $S^\Delta$, the bijection
\[\begin{array}{ccc}
\{g\mycdot x\,|\,x\in D\}&\to&\{h\mycdot x\,|\,x\in D\}\\
x&\mapsto&\frac{h}{g}\mycdot x
\end{array}\]
is continuous for the topology of $\Dloc(S)$, but also of $\D(S^\Delta)$. But since $\{g\mycdot x\,|\,x\in D\}$ is compact, this application is bi-continuous and we obtain the result. Now the proof of lemma is complete.
\end{proof}

\section{Proofs of main results on local Skorokhod metrics}\label{secSkoMet}
In this section we will  prove Theorem \ref{thmLocSkoTop}, Theorem \ref{th2ndCarCvg} and Theorem \ref{thmComp}, by  following the strategy developed in \S 12, pp. 121-137 from \cite{Bi99}. To  construct metrics on $\Dloc(S)$, we will consider a metric $d$ 
on $S$. To begin with, we define two families of pseudo-metrics:
\begin{lemma}[Skorokhod metrics]
For $0\leq t<\infty$ and $K\subset S$ a compact subset, the following two expressions on $\Dexp(S)$:
\begin{samepage}\begin{multline*}
\widetilde{\rho}_{t,K}(x^1,x^2):=\inf_{\substack{t_i\leq\xi(x^i)\\\lambda\in\widetilde{\Lambda},\lambda_{t_1}=t_2}}\sup_{s< t_1}d(x^1_s,x^2_{\lambda_s})\vee\|\lambda-{\rm id}\|_{t_1}\\
\end{multline*}
\vspace{-3.1cm}
\begin{multline*}
\\\vee\max_{i\in\{1,2\}}\Big(d(x^i_{t_i},K^c)\wedge(t-t_i)_+\1_{t_i<\xi(x^i)}\Big),
\end{multline*}
\vspace{-0.5cm}
\begin{multline*}
\rho_{t,K}(x^1,x^2):=\inf_{\substack{t_i\leq\xi(x^i)\\\lambda\in\Lambda,\lambda_{t_1}=t_2}}\sup_{s< t_1}d(x^1_s,x^2_{\lambda_s})\vee\|\log\dot{\lambda}\|_{t_1}\vee\|\lambda-{\rm id}\|_{t_1}\\
\end{multline*}
\vspace{-2.9cm}
\begin{multline*}
\\\vee\max_{i\in\{1,2\}}\Big(d(x^i_{t_i},K^c)\wedge(t-t_i)_+\1_{t_i<\xi(x^i)}\Big).
\end{multline*}\end{samepage}
define two pseudo-metrics.
\end{lemma}
\begin{proof}
Let us perform the proof for $\rho_{t,K}$, the proof being similar for $\widetilde{\rho}_{t,K}$.  The non-trivial part is the triangle inequality. Let $x^1,x^2,x^3\in\Dexp(S)$ and $\eps>0$ be then there are $t_1\leq\xi(x^1)$, $t_2,\widehat{t}_2\leq\xi(x^2)$, $\widehat{t}_3\leq\xi(x^3)$ and $\lambda^1\in\Lambda$, $\lambda^2\in\Lambda$ such that
\begin{samepage}\begin{multline*}
\rho_{t,K}(x^1,x^2)+\eps\geq\sup_{s< t_1}d(x^1_s,x^2_{\lambda^1_s})\vee
\|\log\dot{\lambda}^1\|_{t_1}\vee\|\lambda^1-{\rm id}\|_{t_1}\\
\end{multline*}
\vspace{-2.7cm}
\begin{multline*}
\\\vee\max_{i\in\{1,2\}}\Big(d(x^i_{t_i},K^c)\wedge(t-t_i)_+\1_{t_i<\xi(x^i)}\Big),
\end{multline*}
\vspace{-0.5cm}
\begin{multline*}
\rho_{t,K}(x^2,x^3)+\eps\geq\sup_{s< \widehat{t}_2}d(x^2_s,x^3_{\lambda^2_s})\vee
\|\log\dot{\lambda}^2\|_{\widehat{t}_2}\vee\|\lambda^2-{\rm id}\|_{\widehat{t}_2}\\
\end{multline*}
\vspace{-2.7cm}
\begin{multline*}
\\\vee\max_{i\in\{2,3\}}\Big(d(x^i_{\widehat{t}_i},K^c)\wedge(t-\widehat{t}_i)_+\1_{\widehat{t}_i<\xi(x^i)}\Big).
\end{multline*}\end{samepage}
Define $\widecheck{t}_2:=t_2\wedge\widehat{t}_2$, $\widecheck{t}_1:=(\lambda^1)^{-1}_{\widecheck{t}_2}$, $\widecheck{t}_3:=\lambda^2_{\widecheck{t}_2}$ and $\lambda:=\lambda^2\circ\lambda^1$. Then
\begin{align*}
&\sup_{s< \widecheck{t}_1}d(x^1_s,x^3_{\lambda_s})\leq \sup_{s< t_1}d(x^1_s,x^2_{\lambda^1_s})+\sup_{s< \widehat{t}_2}d(x^2_s,x^3_{\lambda^2_s}),\\
&\|\log\dot{\lambda}\|_{\widecheck{t}_1}\leq \|\log\dot{\lambda}^1\|_{t_1} + \|\log\dot{\lambda}^2\|_{\widehat{t}_2},\quad\|\lambda-{\rm id}\|_{\widecheck{t}_1}\leq \|\lambda^1-{\rm id}\|_{t_1} + \|\lambda^2-{\rm id}\|_{\widehat{t}_2}.
\end{align*}
Moreover, for instance, if $\widecheck{t}_1\not =t_1$, then $\widecheck{t}_1<t_1\leq \xi(x_1)$, $\widehat{t}_2=\widecheck{t}_2<t_2\leq \xi(x_2)$ and
\begin{align*}
d(x^1_{\widecheck{t}_1},K^c)\wedge(t-\widecheck{t}_1)_+
&\leq d(x^1_{\widecheck{t}_1},x^2_{\widecheck{t}_2})\vee|\widecheck{t}_2-\widecheck{t}_1|+d(x^2_{\widecheck{t}_2},K^c)\wedge(t-\widecheck{t}_2)_+\\
&\leq \sup_{s< t_1}d(x^1_s,x^2_{\lambda^1_s})\vee\|\lambda^1-{\rm id}\|_{t_1} + d(x^2_{\widehat{t}_2},K^c)\wedge(t-\widehat{t}_2)_+.
\end{align*}
Hence
\[
\rho_{t,K}(x^1,x^3)\leq \rho_{t,K}(x^1,x^2)+\rho_{t,K}(x^2,x^3)+2\eps,
\]
so letting $\eps\to 0$, we obtain the triangular inequality.
\end{proof}
We prove that these pseudo-metrics are in somehow equivalent:
\begin{lemma}\label{lmEqDist}
Take $x,y\in\Dloc(S)$, $t\geq 0$ and a compact subset $K\subset S$, if $\widetilde{\rho}_{t,K}(x,y)\leq\frac{1}{9}$ then
\begin{align*}
\rho_{t,K}(x,y)\leq 6\cdot\sqrt{\widetilde{\rho}_{t,K}(x,y)}\vee\omega^\prime_{t,K,x}\left(\sqrt{\widetilde{\rho}_{t,K}(x,y)}\right).
\end{align*}
\end{lemma}
\begin{proof}
Let $\eps>0$ be arbitrary. There exist $\mu\in\widetilde{\Lambda}$ and $T\geq 0$ such that $T\leq\xi(x)$, $\mu_T\leq\xi(y)$ and
\begin{align*}
\sup_{s< T}d(x_s,y_{\mu_s})\vee\|\mu-{\rm id}\|_T\leq \widetilde{\rho}_{t,K}(x,y)+\eps ,\\
d(x_T,K^c)\wedge(t-T)_+\1_{T<\xi(x)}\leq \widetilde{\rho}_{t,K}(x,y)+\eps ,\\
d(y_{\mu_T},K^c)\wedge(t-\mu_T)_+\1_{\mu_T<\xi(y)}\leq \widetilde{\rho}_{t,K}(x,y)+\eps.
\end{align*}
Let $\delta>2\widetilde{\rho}_{t,K}(x,y)+2\eps$ be arbitrary, there exist $0=t_0<\cdots< t_N\leq\xi(x)$ such that
\[
\sup_{\substack{0\leq i<N\\t_i\leq s_1,s_2<t_{i+1}}}d(x_{s_1},x_{s_2})\leq \omega^\prime_{t,K,x}(\delta)+\eps,
\]
$\delta<t_{i+1}-t_i\leq 2\delta$ and $(t_N,x_{t_N})\not\in[0,t]\times K$.
Set $n_0:=\max\left\{0\leq i\leq N\mid t_i\leq T\right\}$ and $\widetilde{T}:=t_{n_0}$. Define $\lambda\in\Lambda$ by
\[\left\{\begin{array}{ll}
\forall i\leq n_0,&\lambda_{t_i}=\mu_{t_i},\\
\forall i< n_0,&\lambda\text{ is affine on }[t_i,t_{i+1}],\\
\forall s\geq \widetilde{T},&\dot{\lambda}_s=1.
\end{array}\right.\]
Then
\[
\|\lambda-{\rm id}\| = \sup_{0<i\leq n_0}\|\mu_{t_i}-t_i\|\leq \|\mu-{\rm id}\|_T\leq \widetilde{\rho}_{t,K}(x,y)+\eps.
\]
For  $0\leq i< n_0$ we have 
\[
\left|\frac{\mu_{t_{i+1}}-\mu_{t_i}-t_{i+1}+t_i}{t_{i+1}-t_i}\right|\leq\frac{2\|\mu-{\rm id}\|_T}{\delta}
\leq\frac{2\widetilde{\rho}_{t,K}(x,y)+2\eps}{\delta}<1,
\]
so by the classical estimate:
\[
|\log(1+r)|\leq\frac{|r|}{1-|r|}\quad\mbox{ for }\,|r|<1.
\]
we deduce 
\[
\|\log\dot{\lambda}\| = \sup_{0\leq i< n_0}\left|\log\frac{\mu_{t_{i+1}}-\mu_{t_i}}{t_{i+1}-t_i}\right|\leq\frac{2\widetilde{\rho}_{t,K}(x,y)+2\eps}{\delta-2\widetilde{\rho}_{t,K}(x,y)-2\eps}.
\]
Since for $s<\lambda_{\widetilde{T}}$, $\lambda^{-1}_s$ and $\mu^{-1}_s$ lies  in the same interval $[t_i,t_{i+1})$. Therefore
\[
\sup_{s<\lambda_{\widetilde{T}}}d(x_{\lambda^{-1}_s},y_s)\leq \sup_{s<\lambda_{\widetilde{T}}}\left(d(x_{\mu^{-1}_s},y_s)+d(x_{\mu^{-1}_s},x_{\lambda^{-1}_s})\right)\leq \widetilde{\rho}_{t,K}(x,y)+\omega^\prime_{t,K,x}(\delta)+2\eps.
\]
For the two last terms in $\rho_{t,K}$ we may consider only the case were $\widetilde{T}\not = T$. If $n_0=N$: $d(x_{\widetilde{T}},K^c)\wedge(t-\widetilde{T})_+=0$, otherwise:
\begin{align*}
d(x_{\widetilde{T}},K^c)\wedge(t-\widetilde{T})_+
&\leq d(x_T,K^c)\wedge(t-T)_++d(x_{\widetilde{T}},x_T)\vee(T-\widetilde{T})\\
&\leq\widetilde{\rho}_{t,K}(x,y)+\omega^\prime_{t,K,x}(\delta)\vee (2\delta)+2\eps .
\end{align*}
By using $\lambda_{\widetilde{T}}=\mu_{\widetilde{T}}$, we also have
\begin{align*}
d(y_{\lambda_{\widetilde{T}}},K^c)\wedge(t-\lambda_{\widetilde{T}})_+
&\leq d(x_{\widetilde{T}},K^c)\wedge(t-\widetilde{T})_+ + d(x_{\widetilde{T}},y_{\mu_{\widetilde{T}}})\vee|\widetilde{T}-\mu_{\widetilde{T}}|\\
&\leq 2\widetilde{\rho}_{t,K}(x,y)+\omega^\prime_{t,K,x}(\delta)\vee (2\delta)+3\eps.
\end{align*}
Letting $\eps\to 0$ we obtain that for all $\delta>2\widetilde{\rho}_{t,K}(x,y)$,
\[
\rho_{t,K}(x,y)\leq\left(2\widetilde{\rho}_{t,K}(x,y)+\omega^\prime_{t,K,x}(\delta)\vee (2\delta)\right)\vee\frac{2\widetilde{\rho}_{t,K}(x,y)}{\delta-2\widetilde{\rho}_{t,K}(x,y)}.
\]
Finally, by taking $\delta:= \sqrt{\widetilde{\rho}_{t,K}(x,y)}$ we have for $\widetilde{\rho}_{t,K}(x,y)\leq\frac{1}{9}$
\begin{align*}
\rho_{t,K}(x,y)&\leq\Big(2\widetilde{\rho}_{t,K}(x,y)+\omega^\prime_{t,K,x}(\delta)\vee (2\delta)\Big)\vee\frac{2\widetilde{\rho}_{t,K}(x,y)}{\delta-2\widetilde{\rho}_{t,K}(x,y)}\\
&\leq \Big(\frac{2}{3}\sqrt{\widetilde{\rho}_{t,K}(x,y)}+\omega^\prime_{t,K,x}\big(\sqrt{\widetilde{\rho}_{t,K}(x,y)}\big)\vee \big(2\sqrt{\widetilde{\rho}_{t,K}(x,y)}\big)\Big)\vee 6\sqrt{\widetilde{\rho}_{t,K}(x,y)}\\
&\leq 6\cdot\sqrt{\widetilde{\rho}_{t,K}(x,y)}\vee\omega^\prime_{t,K,x}\big(\sqrt{\widetilde{\rho}_{t,K}(x,y)}\big).\qedhere
\end{align*}
\end{proof}
At this level it can be pointed out that we obtain the definition of the  local Skorokhod topology. Indeed, by using Proposition \ref{propMod}, Lemma \ref{lmEqDist} and the fact that $\widetilde{\rho}_{t,K}\leq \rho_{t,K}$, the two families of pseudo-metrics $(\widetilde{\rho}_{t,K})_{t,K}$ and $(\rho_{t,K})_{t,K}$ define the same topology on $\Dloc(S)$, the local Skorokhod topology.

If $(K_n)_{n\in\N}$ is an exhaustive sequence of compact sets of $S$, then the mapping
\begin{align}\label{eqDist}\begin{array}{ccc}
\Dloc(S)^2&\to&\R_+\\
(x,y)&\mapsto&\sum_{n\in\N}2^{-n}\rho_{n,K_n}(x,y)\wedge 1
\end{array}\end{align}
is a metric for the local Skorokhod topology.
By using a diagonal extraction procedure, it is not difficult to prove that a sequence $(x^k)_k$ converges to $x$ for this topology if and only if there exists a sequence $(\lambda^k)_k$ in $\Lambda$ such that
\begin{itemize}
\item either $\xi(x)<\infty$ and $\{x_s\}_{s<\xi(x)}\Subset S$: for $k$ large enough $\lambda^k_{\xi(x)}\leq\xi(x^k)$ and
\[
\sup_{s<\xi(x)}d(x_s,x^k_{\lambda^k_s})\cv{}0,\hspace{1cm}x^k_{\lambda^k_{\xi(x)}}\cv{}\Delta,\hspace{1cm}\|\log\dot{\lambda}^k\|_{\xi(x)}\cv{}0,\quad\mbox{ as }\,k\to\infty,
\]
\item or $\xi(x)=\infty$ or $\{x_s\}_{s<\xi(x)}\not\Subset S$: for all $t<\xi(x)$, for $k$ large enough $\lambda^k_t<\xi(x^k)$ and
\[
\sup_{s\leq t}d(x_s,x^k_{\lambda^k_s})\cv{}0,\hspace{1cm}\|\log\dot{\lambda}^k\|_t\cv{}0,\quad\mbox{ as }\,k\to\infty.
\]
\end{itemize}
The local Skorokhod topology can be described by a similar characterisation with $\lambda^k\in\Lambda$ replaced by $\lambda^k\in\widetilde{\Lambda}$ and respectively, $\|\log\dot{\lambda}^k\|$ replaced by $\|\lambda -{\rm id}\|$.
The fact that the local Skorokhod topology does not depend on the distance $d$ is a consequence of the following lemma, which states essentially that two metrics on a compact set are uniformly equivalent:
\begin{lemma}\label{lmUU}
Let $T$ be a set and $x,x^k\in S^T$ be such that $\{x_t\}_{t\in T}\Subset S$, then
\[
\sup_{t\in T}d(x_t,x^k_t)\cv{}0,\quad\text{as }k\to\infty,
\]
if and only if
\[
\forall U\subset S^2\text{ open subset containing }\{(y,y)\}_{y\in S},\quad\exists k_0\;\forall k\geq k_0,\;\forall t\in T,\quad (x_t,x^k_t)\in U.
\]
So the topology of the uniform convergence on $\left\{x\in S^T\mid \{x_t\}_{t\in T}\Subset S\right\}$ depends only of the topology of $S$.
\end{lemma}
\begin{proof}
Suppose that $\sup_{t\in T}d(x_t,x^k_t)\cv{}0$ as $k\to\infty$ and take an open subset $U\subset S^2$ containing $\{(y,y)\}_{y\in S}$. By compactness there exists $\eps>0$ such that
\[
\Big\{(y_1,y_2)\in S^2\;|\; y_1\in \{x_t\}_t,~d(y_1,y_2)<\eps\Big\}\subset U,
\]
so for $k$ large enough and for all $t$, $(x_t,x^k_t)\in U$. To get the converse property 
we use the fact that $\big\{(y_1,y_2)\in S^2\,|\,d(y_1,y_2)<\eps\big\}$ is open.
\end{proof}
In the next lemma we discuss the completeness :
\begin{lemma}\label{lmComp}
Suppose that $(S,d)$ is complete. Then any sequence $(x^k)_k\in(\Dloc(S))^\N$ satisfying
\[
\forall t\geq 0,~\forall K\subset S\text{ compact, }\quad \rho_{t,K}(x^{k_1},x^{k_2})\cv{k_1,k_2\to\infty}0,
\]
admits a limit for the local Skorokhod topology.
\end{lemma}
The proof of this lemma follows the same reasoning as the proof of the triangular inequality, the proof of Theorem 12.2 pp. 128-129 from \cite{Bi99} and the proof of Theorem 5.6 pp. 121-122 from \cite{EK86}.
\begin{proof}
It suffices to prove that $(x^k)_k$ have a converging subsequence. By taking, possibly, a subsequence we can suppose that
\begin{align}\label{eqlmComp1}
\forall t\geq 0,~\forall K\subset S\text{ compact, }\quad \sum_{k\geq 0}\rho_{t,K}(x^k,x^{k+1})<\infty.
\end{align}
We split our proof in five steps. 

\noindent
\emph{Step 1: we construct a sequence $(\lambda^k)_k\subset\Lambda$.}
Let $\mu^k\in\Lambda$ and $\widetilde{t}_k\geq 0$ be such that for all $t\geq 0$ and $K\subset S$ compact, we have for $k$ large enough
\begin{align}\label{eqlmComp2}\begin{split}
\widetilde{t}_k\leq\xi(x^k),\quad \mu^k_{\widetilde{t}_k}\leq\xi(x^{k+1}),\\
\sup_{s< \widetilde{t}_k}d(x^k_s,x^{k+1}_{\mu^k_s})\vee\|\log\dot{\mu}^k\|_{\widetilde{t}_k}\vee\|\mu^k-{\rm id}\|_{\widetilde{t}_k}\leq 2\rho_{t,K}(x^k,x^{k+1}),\\
\Big(d(x^k_{\widetilde{t}_k},K^c)\wedge(t-\widetilde{t}_k)_+\1_{\widetilde{t}_k<\xi(x^k)}\Big)\leq 2\rho_{t,K}(x^k,x^{k+1}),\\
\Big(d(x^{k+1}_{\mu^k_{\widetilde{t}_k}},K^c)\wedge(t-\mu^k_{\widetilde{t}_k})_+\1_{\mu^k_{\widetilde{t}_k}<\xi(x^{k+1})}\Big)\leq 2\rho_{t,K}(x^k,x^{k+1}).
\end{split}\end{align}
For all $k\geq 0$ define
\begin{align}\label{eqlmComp3}
t_k:=\bigwedge_{i\geq 0}(\mu^k)^{-1}\circ\cdots\circ(\mu^{k+i-1})^{-1}(\widetilde{t}_{k+i}),
\end{align}
so $t_k\leq\widetilde{t}_k$ and $\mu^k_{t_k}\leq t_{k+1}$. For $k,i\geq 0$
\begin{align*}
\|\log\frac{\d}{\d s}(\mu^{k+i-1}\circ\cdots\circ\mu^k(s))\|_{t_k}\leq \sum_{\ell=k}^{k+i-1}\|\log\dot{\mu}^\ell\|_{\widetilde{t}_\ell},\\
\|\mu^{k+i-1}\circ\cdots\circ\mu^k-{\rm id}\|_{t_k}\leq \sum_{\ell=k}^{k+i-1}\|\mu^\ell-{\rm id}\|_{\widetilde{t}_\ell}
\end{align*}
and for $j\geq i$
\begin{align*}
\|\mu^{k+j-1}\circ\cdots\circ\mu^k-\mu^{k+i-1}\circ\cdots\circ\mu^k\|_{t_k}
&\leq\|\mu^{k+j-1}\circ\cdots\circ\mu^{k+i}-{\rm id}\|_{t_{k+i}}\\
&\leq \sum_{\ell=k+i}^{k+j-1}\|\mu^\ell-{\rm id}\|_{\widetilde{t}_\ell}.
\end{align*}
Using \eqref{eqlmComp1} and \eqref{eqlmComp2} we obtain
\[
\sum_{\ell\geq 0}\|\mu^\ell-{\rm id}\|_{\widetilde{t}_\ell}<\infty,\quad\quad\sum_{\ell\geq 0}\|\log\dot{\mu}^\ell\|_{\widetilde{t}_\ell}<\infty,
\]
so the restriction to $[0,t_k]$ of continuous functions $\mu^{k+i-1}\circ\cdots\circ\mu^k$ converges uniformly to a continuous function. Set
\[\left\{\begin{array}{ll}
\lambda^k_s:=\lim_{i\to\infty}\mu^{k+i-1}\circ\cdots\circ\mu^k(s),&\text{ if }s\leq t_k, \\
\dot{\lambda}^k_s:=1,&\text{ if }s\geq t_k. \\
\end{array}\right.\]
Clearly for $s\leq t_k$, $\lambda^k_s=\lambda^{k+1}\circ\mu^k(s)$. We have
\begin{align}
\label{eqlmComp4}\|\lambda^k-{\rm id}\|\leq&\sup_{i\geq 0}\|\mu^{k+i-1}\circ\cdots\circ\mu^k-{\rm id}\|_{t_k}\leq \sum_{\ell\geq k}\|\mu^\ell-{\rm id}\|_{\widetilde{t}_\ell}<\infty,\\
\nonumber\|\log\dot{\lambda}^k\|&=\sup_{0\leq s_1<s_2\leq t_k}\left|\log\frac{\lambda^k_{s_2}-\lambda^k_{s_1}}{s_2-s_1}\right|\leq\label{eqlmComp5}\sup_{i\geq 0}\|\log\frac{\d}{\d s}(\mu^{k+i-1}\circ\cdots\circ\mu^k(s))\|_{t_k}\\
&\leq \sum_{\ell\geq k}\|\log\dot{\mu}^\ell\|_{\widetilde{t}_\ell}<\infty,
\end{align}
so $\lambda^k\in\Lambda$.\\
\emph{Step 2: we construct a path $x\in\Dexp(S)$.}
For all $k\geq 0$: $\lambda^k_{t_k}=\lambda^{k+1}_{\mu^k_{t_k}}\leq\lambda^{k+1}_{t_{k+1}}$ and moreover for all $0\leq k_1\leq k_2$
\[
\sup_{s<\lambda^{k_1}_{t_{k_1}}}d\left(x^{k_1}_{(\lambda^{k_1})^{-1}_s},x^{k_2}_{(\lambda^{k_2})^{-1}_s}\right)
=\sup_{s<t_{k_1}}d\left(x^{k_1}_s,x^{k_2}_{\mu^{k_1-1}\circ\cdots\circ\mu^{k_1}(s)}\right)
\leq\sum_{\ell=k_1}^{k_2-1}\sup_{s< \widetilde{t}_\ell}d(x^\ell_s,x^{\ell+1}_{\mu^\ell_s}).
\]
By using \eqref{eqlmComp1} and \eqref{eqlmComp2} we get $\sum_{\ell\geq 0}\sup_{s< \widetilde{t}_\ell}d(x^\ell_s,x^{\ell+1}_{\mu^\ell_s})<\infty$.  By using the completeness of $(S,d)$, we deduce that, for each $m\in\N$, the sequence $x^k_{(\lambda^k)^{-1}}$ converges uniformly on $[0,\lambda^m_{t_m})$. We can define $x\in\Dexp(S)$ by setting
\[
\displaystyle\xi(x):=\lim_{k\to\infty}\lambda^k_{t_k}\quad\mbox{ and }\quad
\displaystyle \forall s<\xi(x),~x_s:=\lim_{k\to\infty}x^k_{(\lambda^k)^{-1}_s}.
\]
We see that, for all $k\geq 0$
\begin{align}\label{eqlmComp6}
\sup_{s<\lambda^k_{t_k}}d(x^k_{(\lambda^k)^{-1}_s},x_s)
=\sup_{s<t_k}d\big(x^k_s,x_{\lambda^k_s}\big)
\leq\sum_{\ell\geq k}\sup_{s< \widetilde{t}_\ell}d(x^\ell_s,x^{\ell+1}_{\mu^\ell_s}).
\end{align}
\emph{Step 3: we prove that the infimum in \eqref{eqlmComp3} is a minimum.}
Suppose that there exists $k_0\geq 0$ such that
\[
\forall i\geq 0,\quad t_{k_0}<(\mu^{k_0})^{-1}\circ\cdots\circ(\mu^{k_0+i-1})^{-1}(\widetilde{t}_{k_0+i})
\]
and we will show that one get a contradiction. Firstly, note that for all $k\geq k_0$ we necessarily have $\mu^k_{t_k}=t_{k+1}$ and $t_k<\widetilde{t}_k$ so $\lambda^k_{t_k}$ is constant equal to $\xi(x)$ and furthermore
\[
d(x^k_{t_k},x^{k+1}_{t_{k+1}})\leq\sup_{s< \widetilde{t}_k}d(x^k_s,x^{k+1}_{\mu^k_s}).
\]
Since $\sum_{k\geq 0}\sup_{s< \widetilde{t}_k}d(x^k_s,x^{k+1}_{\mu^k_s})<\infty$,
$x^k_{t_k}$ converges to an element $a\in S$. Let $\eps>0$ be arbitrary such that $K:=\overline{B(a,3\eps)}\subset S$ is compact. Let $k_1\geq k_0$ be such that 
\[
d(x^{k_1}_{t_{k_1}},a)<\eps,\quad \sum_{k\geq k_1}\rho_{\xi(x)+4\eps,K}(x^k,x^{k+1})<\frac{\eps}{2}
\]
and such that \eqref{eqlmComp2} holds for all $k\geq k_1$ with $t=\xi(x)+4\eps$. Set
\[
s_\ell:=\bigwedge_{0\leq i\leq \ell}(\mu^{k_1})^{-1}\circ\cdots\circ(\mu^{k_1+i-1})^{-1}(\widetilde{t}_{k_1+i}).
\]
It is clear that $s_\ell>t_{k_1}$ and $s_\ell$ is a decreasing sequence converging to $t_{k_1}$, so the set $\left\{\ell>0\mid s_\ell<s_{\ell-1}\right\}$ is infinite. Let $\ell>0$ be such that $s_\ell<s_{\ell-1}$ and $s_\ell-t_{k_1}<\eps$, then
\[
s_\ell=(\mu^{k_1})^{-1}\circ\cdots\circ(\mu^{k_1+\ell-1})^{-1}(\widetilde{t}_{k_1+\ell}).
\]
Therefore
\begin{align}\label{eqlmComp8}
\widetilde{t}_{k_1+\ell}&=\mu^{k_1+\ell-1}\circ\cdots\circ\mu^{k_1}(s_\ell)<\mu^{k_1+\ell-1}\circ\cdots\circ\mu^{k_1}(t_{k_1}+\eps)\\
\nonumber&\leq \sum_{i=k_1}^{k_1+\ell-1}\|\mu^i-{\rm id}\|_{\widetilde{t}_i}+t_{k_1}+\eps
\leq \sum_{i=k_1}^{k_1+\ell-1}\|\mu^i-{\rm id}\|_{\widetilde{t}_i}+\|\lambda^{k_1}-{\rm id}\|_{t_{k_1}}+\xi(x)+\eps\\
\nonumber&\leq \xi(x)+\eps+2\sum_{i\geq k_1}\|\mu^i-{\rm id}\|_{\widetilde{t}_i}
\leq \xi(x)+\eps+4\sum_{i\geq k_1}\rho_{\xi(x)+4\eps,K}(x^i,x^{i+1})
<\xi(x)+3\eps.
\end{align}
Furthermore $\widetilde{t}_{k_1+\ell}<\mu^{k_1+\ell-1}_{\widetilde{t}_{k_1+\ell-1}}\leq \xi(x^{k_1+\ell})$ and 
\[
d(x^{k_1+\ell}_{\widetilde{t}_{k_1+\ell}},K^c)\wedge(\xi(x)+4\eps-\widetilde{t}_{k_1+\ell})_+\1_{\widetilde{t}_{k_1+\ell}<\xi(x^{k_1+\ell})}\leq 2\rho_{\xi(x)+4\eps,K}(x^{k_1+\ell},x^{k_1+\ell+1})< \eps,
\]
so by \eqref{eqlmComp8}
\[
d(x^{k_1+\ell}_{\widetilde{t}_{k_1+\ell}},K^c)<\eps,
\]
and
\[
d(x^{k_1}_{s_\ell},K^c)\leq d(x^{k_1}_{s_\ell},x^{k_1+\ell}_{\widetilde{t}_{k_1+\ell}})+d(x^{k_1+\ell}_{\widetilde{t}_{k_1+\ell}},K^c)< \sum_{i=k_1}^{k_1+\ell-1}\sup_{s< \widetilde{t}_i}d(x^i_s,x^{i+1}_{\mu^i_s}) +\eps <2\eps.
\]
Hence we have $d(a,x^{k_1}_{s_\ell})>\eps$ and $d(a,x^{k_1}_{t_{k_1}})<\eps$. Letting tend $\ell\to\infty$ we get a contradiction.\\
\emph{Step 4: fix $t\geq 0$ and $K\subset S$  a compact set: we prove that $\lim_{k\to\infty}\rho_{t,K}(x^k,x)=0$.}
Taking $k$ large enough \eqref{eqlmComp2} holds and by using \eqref{eqlmComp4}, \eqref{eqlmComp5} and \eqref{eqlmComp6},
\begin{align*}
\sup_{s<t_k}d\big(x^k_s,x_{\lambda^k_s}\big)\leq\sum_{\ell\geq k}\sup_{s< \widetilde{t}_\ell}d(x^\ell_s,x^{\ell+1}_{\mu^\ell_s})\leq 2\sum_{\ell\geq k}\rho_{t,K}(x^\ell,x^{\ell+1}),\\
\|\lambda^k-{\rm id}\|\leq\sum_{\ell\geq k}\|\mu^\ell-{\rm id}\|_{\widetilde{t}_\ell}\leq 2\sum_{\ell\geq k}\rho_{t,K}(x^\ell,x^{\ell+1}),\\
\|\log\dot{\lambda}^k\|\leq \sum_{\ell\geq k}\|\log\dot{\mu}^\ell\|_{\widetilde{t}_\ell}\leq 2\sum_{\ell\geq k}\rho_{t,K}(x^\ell,x^{\ell+1}).
\end{align*}
Moreover by the previous step we know that the infimum in \eqref{eqlmComp3} is a minimum and so, in the present step, we set
\begin{align*}
&m:=\min\left\{\ell\geq k\mid (\mu^k)^{-1}\circ\cdots\circ(\mu^{\ell-1})^{-1}(\widetilde{t}_{\ell})=t_k\right\}\in\N,\\
&M:=\sup\left\{\ell\geq k\mid (\mu^k)^{-1}\circ\cdots\circ(\mu^{\ell-1})^{-1}(\widetilde{t}_{\ell})=t_k\right\}\in\N\cup\{\infty\}.
\end{align*}
Then
\begin{align*}
d(x^k_{t_k},K^c)\wedge(t-t_k)_+\1_{t_k<\xi(x^k)}
&\leq \sum_{\ell=k}^{m-1}\sup_{s< \widetilde{t}_\ell}d(x^\ell_s,x^{\ell+1}_{\mu^\ell_s})\vee\|\mu^\ell-{\rm id}\|_{\widetilde{t}_\ell}\\
&\quad+d(x^{m}_{\widetilde{t}_{m}},K^c)\wedge(t-\widetilde{t}_{m})_
+\1_{\widetilde{t}_{m}<\xi(x^{m})}\\
&\leq 2\sum_{\ell\geq k}\rho_{t,K}(x^\ell,x^{\ell+1}).
\end{align*}
It is clear that $\lambda^k_{t_k}=\xi(x)$ if and only if $M=\infty$. If $M<\infty$
\begin{align}
\nonumber d(x_{\lambda^k_{t_k}},K^c)\wedge(t-\lambda^k_{t_k})_+
&=d(x_{\lambda^{M}_{t_{M}}},K^c)\wedge(t-\lambda^{M}_{t_{M}})_+\\
\nonumber&\leq d(x^{M+1}_{\mu^{M}_{t_{M}}},K^c)\wedge(t-\mu^{M}_{t_{M}})_+ + \sum_{\ell> M}\sup_{s< \widetilde{t}_\ell}d(x^\ell_s,x^{\ell+1}_{\mu^\ell_s})\vee\|\mu^\ell-{\rm id}\|_{\widetilde{t}_\ell}\\
\label{eqlmComp7}&\leq 2\sum_{\ell\geq k}\rho_{t,K}(x^\ell,x^{\ell+1}).
\end{align}
We have proved that
\[
\rho_{t,K}(x^k,x)\leq 2\sum_{\ell\geq k}\rho_{t,K}(x^\ell,x^{\ell+1}) \cv{k\to\infty}0.
\]
\emph{Step 5: we prove that $x\in\Dloc(x)$.}
Suppose that $\xi(x)<\infty$ and that $\{x_s\}_{s<\xi(x)}\Subset S$. 
Let $\eps >0$ be such that $K:=\left\{y\in S\mid d(y,\{x_s\}_{s<\xi(x)})\leq \eps\right\}$ is compact and set $t=\xi(x)+\eps$. By using \eqref{eqlmComp7} we have, for $k$ large enough,
\[
d(x_{\lambda^k_{t_k}},K^c)\wedge(t-\lambda^k_{t_k})_+\1_{\lambda^k_{t_k}<\xi(x)}\leq 2\sum_{\ell\geq k}\rho_{t,K}(x^\ell,x^{\ell+1})<\eps.
\]
Then $\xi(x) = \lambda^k_{t_k}$ and  we deduce that
\[
\sup_{s<\xi(x)}d\big(x^k_{(\lambda^k)^{-1}s},x_s\big)\leq 2\sum_{\ell\geq k}\rho_{t,K}(x^\ell,x^{\ell+1})\cv{k\to\infty} 0
\]
and that the limit $x_{\xi(x)-}$ exists in $S$. Therefore $x\in\Dloc(S)$ and $x^k$ converges to $x$ for the local Skorokhod topology.
\end{proof}
To prove the separability and the criterion of compactness we will use the following technical result:
\begin{lemma}\label{lmAproxEsp}
Let $R\subset S$, $\delta>0$ and $N\in\N$ be. Define
\begin{multline*}
\Ec_{R,\delta,N}:=\big\{x\in\Dloc(S)\,\big|\, \xi(x)\leq N\delta,~\forall k\in\N,~x\text{ is a constant in }R\cup\{\Delta\}\\\text{ on }[k\delta,(k+1)\delta)\big\}.
\end{multline*}
Then for any $x\in\Dloc(S)$
\[
\widetilde{\rho}_{N\delta ,K}(x,\Ec_{R,\delta,N})\leq \Big(\sup_{a\in K}d(a,R)+ \omega^\prime_{N\delta,K,x}(\delta)\Big)\vee\delta.
\]
\end{lemma}
\begin{proof}
For arbitrary $\eps>0$ , there exist $0=t_0<\cdots< t_M\leq\xi(x)$ such that
\[
\sup_{\substack{0\leq i<M\\t_i\leq s_1,s_2<t_{i+1}}}d(x_{s_1},x_{s_2})\leq \omega^\prime_{N\delta,K,x}(\delta)+\eps,
\]
such that for all $0\leq i <M$, $t_{i+1}>t_i+\delta$ and $(t_M,x_{t_M})\not\in[0,N\delta]\times K$.
Denote $t^*:=\min\big\{s\geq 0\,\big|\, s\geq N\delta\text{ or }d(x_s,K^c)=0\big\}\leq t_M$ and define $\widetilde{M}:=min\left\{0\leq i\leq M\mid t_i\geq t^*\right\}$. Define $\widetilde{t}_{\widetilde{M}}:=\left\lceil\frac{t^*}{\delta}\right\rceil\delta$ where $\lceil r\rceil$ denotes the smallest integer larger or equal than the real number $r$. Moreover, for $0\leq i <\widetilde{M}$ define $\widetilde{t}_i:=\left\lfloor\frac{t_i}{\delta}\right\rfloor\delta$ where we recall that $\lfloor r\rfloor$ denotes the integer part of the real number $r$, so $0=\widetilde{t}_0<\ldots<\widetilde{t}_{\widetilde{M}}$. Finally, we define $\widetilde{x}\in \Ec_{R,\delta,N}$ by
\[\left\{\begin{array}{l}
\xi(\widetilde{x}):=\widetilde{t}_{\widetilde{M}},\\
\forall 0\leq i <\widetilde{M}:\quad\text{we choose }\widetilde{x}_{\widetilde{t}_i}\text{ in $R$ such that }d(x_{t_i},\widetilde{x}_{\widetilde{t}_i})< d(x_{t_i},R)+\eps,\\
\forall 0\leq i <\widetilde{M},~\forall\widetilde{t}_i\leq s <\widetilde{t}_{i+1}:\quad\widetilde{x}_s:=\widetilde{x}_{\widetilde{t}_i},
\end{array}\right.\]
and $\lambda\in\widetilde{\Lambda}$ given by
\[\left\{\begin{array}{l}
\forall 0\leq i \leq\widetilde{M}:\quad\lambda_{t_i\wedge t^*}=\widetilde{t}_i,\\
\forall 0\leq i <\widetilde{M}:\quad\lambda\text{ is affine on }[t_i,t_{i+1}\wedge t^*],\\
\forall s\geq t^*:\quad\dot{\lambda}_s = 1.\\
\end{array}\right.\]
We can write
\begin{align*}
\widetilde{\rho}_{N\delta ,K}(x,\Ec_{R,\delta,N})&\leq \widetilde{\rho}_{N\delta ,K}(x,\widetilde{x})
\leq \sup_{s< t^*}d(x_s,\widetilde{x}_{\lambda_s})\vee\|\lambda-{\rm id}\|\\
&\leq \left(\sup_{a\in K}d(a,R)+ \omega^\prime_{t,K,x}(\delta)+2\eps\right)\vee\delta,
\end{align*}
so letting $\eps\to0$ we obtain the result.
\end{proof}
The separability is an easy consequence:
\begin{lemma}
The local Skorokhod topology on $\Dloc(S)$ is separable.
\end{lemma}
\begin{proof}
Let $R$ be a countable dense part of $S$ and introduce the countable set
\[
E:=\bigcup_{n,N\in\N^*}\Ec_{R,\frac{1}{n},N}.
\]
Consider $x\in\Dloc(S)$, $t\geq 0$, $K\subset S$ a compact set and let $\eps >0$ be. We choose $n\in\N^*$ such that $n^{-1}\leq \eps$ and $\omega^\prime_{t+1,K,x}(n^{-1})\leq \eps$ and set $N:=\lceil n t\rceil$. We can write
\begin{align*}
\widetilde{\rho}_{t ,K}(x,\Ec_{R,\frac{1}{n},N})&\leq \widetilde{\rho}_{\frac{N}{n} ,K}(x,\Ec_{R,\frac{1}{n},N})\leq \left(\sup_{a\in K}d(a,R)+ \omega^\prime_{\frac{N}{n},K,x}(\frac{1}{n})\right)\vee\frac{1}{n}\\
&\leq\omega^\prime_{t+1,K,x}(\frac{1}{n})\vee\frac{1}{n}\leq \eps.
\end{align*}
We deduce that $E$ is dense, hence $\Dloc(S)$ is separable.
\end{proof}
We have now all the ingredients to prove the characterisation of the compactness:
\begin{proof}[Proof of Theorem \ref{thmComp}]
First, notice that, similarly as in the proof of Lemma \ref{lmUU}, the condition \eqref{eqThmComp} is equivalent to:
for all $t\geq 0$, all compact subset $K\subset S$ and all open subset $U\subset S^2$ containing the diagonal $\big\{(y,y)\,|\, y\in S\big\}$, there exists $\delta>0$ such that for 
all $x\in D$ there exist $0=t_0<\cdots <t_N\leq\xi(x)$ such that
\[
\forall 0\leq i<N,~s_1,s_2\in[t_i,t_{i+1}),\quad (x_{s_1},x_{s_2})\in U,
\]
for all $0\leq i<N$, $t_{i+1}-t_i>\delta$, and $(t_N,x_{t_N})\not\in[0,t]\times K$. Hence the condition \eqref{eqThmComp} is independent to $d$ and we can suppose that $(S,d)$ is complete.
Suppose that $D$ satisfy condition \eqref{eqThmComp}, then, by using Lemma \ref{lmComp}, we need to prove that for all $t\geq 0$, $K\subset S$ a compact set and $\eps>0$ arbitrary, $D$ can be  recovered by  a finite number of $\rho_{t,K}$-balls of radius $\eps$. Let $0<\eta\leq\frac{1}{9}$ be such that
\[
6\cdot\sqrt{\eta}\vee\sup_{x\in D}\omega^\prime_{t,K,x}\left(\sqrt{\eta}\right)\leq \eps,
\]
and let $\delta\leq\eta$ be such that
\[
\sup_{x\in D}\omega^\prime_{t+1,K,x}\left(\delta\right)\leq\frac{\eta}{2}.
\]
Since $K$ is compact we can choose a finite set $R\subset S$ such that
\[
\sup_{a\in K}d(a,R)\leq\frac{\eta}{2},
\]
take $N:=\lceil t\delta^{-1}\rceil$.  Then by using Lemma \ref{lmAproxEsp},
\begin{align*}
\sup_{x\in D}\widetilde{\rho}_{t ,K}(x,\Ec_{R,\delta,N})\leq \sup_{x\in D}\widetilde{\rho}_{N\delta ,K}(x,\Ec_{R,\delta,N})\leq \left(\sup_{a\in K}d(a,R)+ \sup_{x\in D}\omega^\prime_{N\delta,K,x}(\delta)\right)\vee\delta\leq \eta
\end{align*}
and by using Lemma \ref{lmEqDist},
\[
\sup_{x\in D}\rho_{t ,K}(x,\Ec_{R,\delta,N})\leq 6\sup_{x\in D}\left(\sqrt{\widetilde{\rho}_{t ,K}(x,\Ec_{R,\delta,N})}\vee\omega^\prime_{t,K,x}\left(\sqrt{\widetilde{\rho}_{t ,K}(x,\Ec_{R,\delta,N})}\right)\right)\leq \eps.
\]
Since $\Ec_{R,\delta,N}$ is finite we can conclude that $D$ is relatively compact.\\
To prove the converse sentence, thanks to the first part of Proposition \ref{propMod} it is enough to prove that if $x^k,x\in\Dloc(S)$ with $x^k$ converging to $x$, then for all $t\geq 0$ and all compact subset $K\subset S$,
\[
\limsup_{k\to\infty}\omega^\prime_{t,K,x^k}(\delta)\cv{\delta\to 0}0.
\]
This is a direct consequence of Proposition \ref{propMod}. Let us stress that although we cite Theorem \ref{thmLocSkoTop} in the proof of Proposition \ref{propMod}, in reality we only need the sequential characterisation of the convergence.
\end{proof}
We close this section by the study of the Borel $\sigma$-algebra $\Bc(\Dloc(S))$.
\begin{lemma}
Borel $\sigma$-algebra $\Bc(\Dloc(S))$ coincides with $\Fc$.
\end{lemma}
\begin{proof}
Let $f\in\rmC(S^\Delta)$ and $0\leq a<b<\infty$ be. Consider $x^k\in\Dloc(S)$ converging to $x\in\Dloc(S)$, with $\xi(x)>b$, and take $\lambda^k\in\Lambda$ as in Theorem \ref{thmLocSkoTop}. Then for $k$ large enough $b\vee\lambda^k_b<\xi(x^k)$ and by dominated convergence
\[
\int_a^bf(x^k_s)\d s = \int_a^{\lambda^k_a}f(x^k_s)\d s+\int_a^bf(x^k_{\lambda^k_s})\dot{\lambda}^k_s\d s+\int_{\lambda^k_b}^bf(x^k_s)\d s\cv{k\to\infty}\int_a^bf(x_s)\d s.
\]
Hence the set $\left\{x\in\Dloc(S)\mid b<\xi(x)\right\}$ is open and on this set the function 
\[
x\mapsto\int_a^bf(x_s)\d s
\] 
is continuous, so for $t\geq 0$ and $\eps>0$ the mapping from $\Dloc(S)$ to $\R$
\[
x\mapsto\left\{\begin{array}{ll}
\frac{1}{\eps}\int_t^{t+\eps}f(x_s)\d s, &\text{if }t+\eps<\xi(x),\\
f(\Delta), &\text{otherwise},\\
\end{array}\right.\]
is measurable for the Borel $\sigma$-algebra $\Bc(\Dloc(S))$ and, the 
same is true for the mapping  $x\mapsto f(x_t)$, by taking the limit. Since $f$ is arbitrary, $x\mapsto x_t$ is also measurable and so $\Fc\subset\Bc(\Dloc(S))$.

Conversely, since the space is separable, it is enough to prove that for each $x^0\in\Dloc(S)$, $t\geq 0$, $K\subset S$ compact and $\eps> 0$ there exists $V\subset \Dloc(S)$, $\Fc$-measurable, such that
\begin{align}\label{eq1lmTrib}
\left\{x\in\Dloc(S)\mid \widetilde{\rho}_{t,K}(x,x^0)\leq \eps\right\}\subset V\subset\left\{x\in\Dloc(S)\mid \widetilde{\rho}_{t,K}(x,x^0)\leq 3 \eps\right\}.
\end{align}
Proposition \ref{propMod} allows to get the existence of $0=t^0_0<\cdots <t^0_N\leq\xi(x^0)$ such that
\[
\sup_{\substack{0\leq i<N\\t^0_i\leq s_1,s_2<t^0_{i+1}}}d(x^0_{s_1},x^0_{s_2})\leq\eps,
\]
and $(t^0_N,x^0_{t^0_N})\not\in[0,t]\times\in K$.
If we define
\[
V:=\left\{x\in\Dloc(S)\mid\begin{array}{l}
\exists0=t_0\leq\cdots\leq t_M\leq \xi(x),~M\leq N\text{ such that:}\\
\hspace{1cm}\forall 0\leq i\leq M,\quad |t_i-t^0_i|\leq \eps\\
\hspace{1cm}\forall 0\leq i< M,~\forall t\in[t_i,t_{i+1}),\quad d(x_t,x^0_{t^0_i})\leq 2\eps\\
\hspace{1cm}d(x^0_{t^0_M},K^c)\wedge (t-t^0_M)_+\1_{t^0_M<\xi(x^0)}\leq 2\eps\\
\hspace{1cm}d(x_{t_M-},K^c)\wedge d(x_{t_M},K^c)\wedge (t-t_M)_+\1_{t_M<\xi(x)}\leq 3\eps\\
\end{array}\right\},
\]
it is straightforward to obtain \eqref{eq1lmTrib}. Since
\[
V=\left\{x\in\Dloc(S)\mid\begin{array}{l}
\forall\delta>0,~\exists0=q_0\leq\cdots\leq q_M< \xi(x)-\delta,~M\leq N\text{ such that:}\\
\hspace{1cm}\forall 0\leq i\leq M,\quad |q_i-t^0_i|\leq \eps+\delta\\
\hspace{1cm}\forall 0\leq i< M,~\forall q\in[q_i+\delta,q_{i+1}-\delta],\quad d(x_q,x^0_{t^0_i})\leq 2\eps\\
\hspace{1cm}d(x^0_{t^0_M},K^c)\wedge (t-t^0_M)_+\1_{t^0_M<\xi(x^0)}\leq 2\eps\\
\hspace{1cm}d(x_{q_M},K^c)\wedge (t-q_M)_+\1_{q_M<\xi(x)}\leq 3\eps+\delta\\
\end{array}\right\},
\]
where $q$, $q_i$ and $\delta$ are chosen to  be rational, $V$ belongs $\Fc$. The proof is now complete.
\end{proof}
\bibliographystyle{alpha}
\bibliography{Loc_Fel.bib}
\end{document}